\documentclass[12pt]{amsart}

\usepackage{tikz}
\usepackage{amssymb}
\usepackage{comment}
\usepackage[all]{xy}
\usepackage{todonotes}

\usepackage{color}
\newcommand{\blue}[1]{\textcolor{blue}{#1}}
\newcommand{\red}[1]{\textcolor{red}{#1}}

\usepackage[bookmarks=false,colorlinks,linkcolor=blue,citecolor=green]{hyperref}

\date{\today}

\theoremstyle{plain}
\newtheorem{theorem}{Theorem}[section]

\newtheorem{lemma}[theorem]{Lemma}
\newtheorem{proposition}[theorem]{Proposition}
\newtheorem{conjecture}[theorem]{Conjecture}

\theoremstyle{definition}
\newtheorem{example}[theorem]{Example} 
\newtheorem{remark}[theorem]{Remark} 
\newtheorem{definition}[theorem]{Definition} 

\theoremstyle{remark}
\newtheorem*{associativity}{Associativity}
\newtheorem*{compatibility}{Compatibility}

\specialcomment{moredetails}{\begingroup\color{red}}{\endgroup}

\def\id{\mathrm{id}}
\def\Id{{\bf 1}}

\let\map=\xrightarrow
\newcommand{\xyinc}{\ar@{^{(}->}}

\def\field{\Bbbk}

\newcommand{\rL}{{\mathbf l}}
\newcommand{\rG}{{\mathbf g}}

\newcommand{\bL}{\mathbf L}

\newcommand{\bH}{\mathbf H}
\newcommand{\bh}{\mathbf h}

\newcommand{\bPi}{\mathbf \Pi}
\newcommand{\bG}{\mathbf{G}}

\newcommand{\Kc}{\mathcal{K}}
\newcommand{\Kcb}{\overline{\Kc}}

\newcommand{\qand}{\quad\text{and}\quad}

\newcommand{\qqand}{\qquad\text{and}\qquad}

 \author{Carolina Benedetti}\address[Benedetti]
{Department of Mathematics and Statistics\\ York  University\\ To\-ron\-to, Ontario M3J 1P3\\ CANADA}

\author{Nantel Bergeron}\address[Bergeron]
{Department of Mathematics and Statistics\\ York  University\\ To\-ron\-to, Ontario M3J 1P3\\ CANADA}
\email{bergeron@mathstat.yorku.ca}
\urladdr{http://www.math.yorku.ca/bergeron}
\thanks{Bergeron supported in part by NSERC}

\title[{A}ntipode for linearized Hopf monoid]{{C}ancelation free formula for the antipode of linearized Hopf monoid}



\begin{document}

\begin{abstract}
Many combinatorial Hopf algebras $H$ in the literature are the functorial image of a linearized Hopf monoid $\bH$. That is, $H=\Kc (\bH)$ or $H=\Kcb (\bH)$. Unlike the functor $\Kcb$, the functor $\Kc$ applied to $\bH$ may not preserve the antipode of $\bH$. In this case, one needs to consider  the larger Hopf monoid $\bL\times\bH$ to get $H=\Kc (\bH)=\Kcb(\bL\times\bH)$ and study the antipode in $\bL\times\bH$. One of the main results in this paper provides a cancelation free and multiplicity free formula for the antipode of $\bL\times\bH$. From this formula we obtain a new antipode formula for $H$. 
We also explore the case when $\bH$ is commutative and cocommutative. In this situation we get new antipode formulas that despite of not being cancelation free, 
can be used to obtain one for $\Kcb(\bH)$ in some cases.
We recover as well many of the well-known cancelation free formulas in the literature. 
One of our formulas for computing the antipode in $\bH$ involves acyclic orientations of hypergraphs as the central tool.
In this vein, we obtain polynomials analogous to the chromatic polynomial of a graph, and also  identities parallel to Stanley's (-1)-color theorem.
One of our examples introduces a \emph{chromatic} polynomial for permutations which counts increasing sequences of the permutation satisfying a pattern. We also study the statistic obtained after evaluating such polynomial at $-1$. Finally, we sketch $q$ deformations and geometric interpretations of our results. This last part will appear in a sequel paper in joint work with J. Machacek.
\end{abstract}

\maketitle

\section*{Introduction}

Computing antipode formulas in various graded Hopf algebras is a classical yet difficult problem to solve. Recently, numerous results in this direction have been provided for various families of Hopf algebras \cite{Humpert-Martin, Benedetti-Sagan, Bergeron-Ceballos, Aguiar-Ardila, Darij-Reiner, BakerJarvisBergeronThiem}.  A motivation to find such formulas lies in their potential geometric interpretation (see for example \cite{Aguiar-Ardila}), or in their use to derive information regarding combinatorial invariants of the discrete objects in play. One example of this is the Hopf algebra of graphs $\mathcal G$ (see, for instance~\cite{Humpert-Martin}). In~\cite{Humpert-Martin} the authors derive the antipode formula and use it to obtain the celebrated Stanley's $(-1)$-color theorem:  the chromatic polynomial of a graph evaluated at $-1$ is, up to a sign, the number of acyclic orientations of the graph. On the geometric side, a remarkable result in\cite{Aguiar-Ardila} shows that such antipode is encoded in the $f$-vector of the graphical zonotope corresponding to the given graph.

The general principle is that antipode formulas provide interesting identities for the combinatorial invariants.
One of the key results in the theory of \emph{Combinatorial Hopf algebras (CHAs)} gives us a canonical way of constructing combinatorial invariants with values in the space $QSym$ of quasisymmetric functions (see~\cite{ABS}).
Letting $H=\bigoplus_{n\ge 0} H_n$ be a CHA with \emph{character} $\zeta\colon H\to\field$
we have a unique Hopf morphism
  $$\Psi:H\to QSym$$
  such that $\zeta=\phi_1\circ\Psi$ where $\phi_1\big(f(x_1,x_2,\ldots)\big)=f(1,0,0,\ldots)$. Moreover, there is a Hopf
  morphism $\phi_t \colon QSym\to \field[t]$ given by 
   $\phi_t (M_a)={t \choose \ell(a)},$
   where $M_a$ is the monomial quasisymmetric function indexed by an integer composition $a=(a_1,a_2,\ldots,a_\ell)$ and $\ell(\alpha)=\ell$ is the length of the composition.
 This Hopf morphism has the property that 
   $$\phi_t\big(f(x_1,x_2,\ldots)\big) \Big|_{t=1}=\phi_1(f)\,.$$
In particular
   $$\phi_t\circ\Psi\Big|_{t=1}=\left( \phi_t\big|_{t=1}\right)\circ\Psi=\phi_1\circ\Psi=\zeta.$$
   
In the case when $H=\mathcal G$ one can define the character
  $$\zeta(G)=\begin{cases}
     1&\text{if $G$ is discrete graph,}\\
     0&\text{otherwise.}
  \end{cases}$$
which gives us, as shown in \cite[Example 4.5]{ABS} that $\chi_G(t)=\phi_t\circ \Psi(G)$ is the chromatic polynomial of $G$.
Stanley's $(-1)$-theorem can be obtained pointing out that
the antipode of $\field[t]$ is given by $S\big(p(t)\big)=p(-t)$. Hence
  $$\chi_G(-1)=S\circ\phi_t\circ \Psi(G)\Big|_{t=1} = \phi_t\circ \Psi\circ S(G)\Big|_{t=1}=\zeta\circ S(G)$$
Using the fact that the discrete graph has coefficient $(-1)^{n}a(G)$ in $S(G)$ (see~\cite{Aguiar-Ardila,Benedetti-Sagan,Humpert-Martin}), where $n$ is the number of vertices of $G$ and $a(G)$ counts the acyclic orientations in it, we see that $\chi_G(-1)=\zeta\circ S(G)=(-1)^{n}a(G)$.

Here, we  present a general framework that allows us to derive new formulas for the antipode of many of the graded Hopf algebras in the literature. 
To achieve this, we lift the structure of graded Hopf algebra to 
a Hopf monoid in Joyal's category of species.
Combinatorial objets which compose and decompose often give rise
to Hopf monoids. These objects are the subject of~\cite[Part~II]{AM:2010}.
The few basic notions and examples needed for our purposes are reviewed in
Section~\ref{s:hopf}, including the Hopf monoids of linear orders $\bL$, the notion of linearized Hopf monoid and the  Hadamard product.

The first goal of this paper is to construct a cancelation free and multiplicity free formula for
the antipode of the Hadamard product $\bL\times\bH$ where $\bH$ is a linearized Hopf monoid. 
This surprising result will be done in Section~\ref{s:antiLxH}. One interesting fact is that even if at the level of Hopf monoids the formula is cancelation free,
many cancelations may occur when applying $\Kcb(\bL\times\bH)$. Yet this gives us new formulas for antipodes and potentially new identities.
We discuss this in Section~\ref{s:KLxH}.

Our next task, in Section~\ref{s:antiH}, is to consider the antipode formula for commutative and cocommutative linearized Hopf monoid $\bH$.
This case is especially interesting as many of the Hopf monoids in combinatorics fall into this class.
One consequence of our analysis is that the most interesting case to consider is the Hopf monoid of hypergraphs $\bf HG$ as defined in Section~\ref{ss:hypergraph}.
The Hopf monoid $\bf HG$ contains all the information to compute antipodes for any other commutative and cocommutative linearized Hopf monoid $\bH$.
This is an interesting fact and we will show completely the relationship.
We give two antipode formulas for $\bH$. One derived in Section~\ref{ss:antiHwithLxH} from our work in Section~\ref{s:antiLxH} and one in Section~\ref{subsec:acyclic} related to orientations of hypergraphs. Applications of this computations are presented in Section~\ref{ss:anticom}.
In Section~\ref{ss:antiStan} we derive combinatorial identities using our antipode formulas.
In particular we introduce a chromatic polynomial for total orders (permutations) and show an analogous to Stanley's $(-1)$-theorem.

In future work, with J. Machacek, we will show a geometric interpretation for the antipode of $\bf HG$ as encoded by a hypergraphical nestohedron. This will done in the spirit of the work in \cite{Aguiar-Ardila}.
We will also investigate $q$-deformations of Hopf structures studied here, leading to a Shareshian-Wach generalization of chromatic quasisymmetric functions. A preview of this work is given in Section~\ref{s:Future}.

\section{Hopf monoids}\label{s:hopf}

We review basic notions on Hopf monoids and illustrate definitions with  three classical examples. We encourage the reader to see \cite{AM:2010} for a deeper study on this topic. We define the notion of {\sl linearized Hopf monoid} as given in \cite{AM:2010,MarbergLin}.

\subsection{Species and Hopf monoids}\label{ss:sp-hopf}
A \emph{vector species} $\bH$ is a collection of vector spaces $\bH[I]$,
one for each finite set $I$, equivariant with respect to bijections $I\cong J$.
A morphism of species $f:\bH\to\bf Q$ is a collection of linear maps
$f_I:\bH[I]\to{\bf Q}[I]$ which commute with bijections.

A \emph{set composition} of a finite set $I$ is a finite sequence $(A_1,\ldots,A_k)$
of disjoint subsets of $I$ whose union is $I$. In this situation, we write
$
(A_1,\ldots,A_k)\models I.
$

A \emph{Hopf monoid} consists of a vector species $\bH$
equipped with two collections $\mu$ and $\Delta$ of linear maps
\[
\bH[A_1]\otimes\bH[A_2] \map{\mu_{A_1,A_2}}\bH[I]
\qand
\bH[I]\map{\Delta_{A_1,A_2}} \bH[A_1]\otimes\bH[A_2]
\]
subject to a number of axioms, of which the main ones follow.

\begin{associativity}
For each set composition $(A_1, A_2, A_3)\models I$, the diagrams
\begin{gather}\label{e:assoc}
\begin{gathered}
\xymatrix@R+1pc@C+20pt{
\bH[A_1]\otimes\bH[A_2]\otimes\bH[A_3]\ar[r]^-{\id\otimes\mu_{A_1,A_2}}
\ar[d]_{\mu_{A_1,A_2}\otimes\id} & \bH[A_1]\otimes\bH[A_2\cup A_3]\ar[d]^{\mu_{A_1,A_2\cup A_3}} \\
\bH[A_1\cup A_2]\otimes \bH[A_3]\ar[r]_-{\mu_{A_1\cup A_2,A_3}} & \bH[I]
}
\end{gathered}
\\
\label{e:coassoc}
\begin{gathered}
\xymatrix@R+1pc@C+40pt{
\bH[I]\ar[r]^-{\Delta_{A_1\cup A_2,A_3}} \ar[d]_{\Delta_{A_1,A_2\cup A_3}}
& \bH[A_1\cup A_2]\otimes \bH[A_3]\ar[d]^{\Delta_{A_1,A_2}\otimes\id}
\\
\bH[A_1]\otimes\bH[A_2\cup A_3]\ar[r]_-{\id\otimes\Delta_{A_1,A_2}}
& \bH[A_1]\otimes\bH[A_2]\otimes\bH[A_3]
}
\end{gathered}
\end{gather}
commute.
\end{associativity}

\begin{compatibility}
Fix two set compositions $(A_1,A_2)$ and $(B_1,B_2)$ of $I$,
and consider the resulting pairwise intersections:
\[
P:=A_1\cap B_1,\ Q:=A_1\cap B_2,\ R:=A_2\cap B_1,\ T:=A_2\cap B_2,
\]
as illustrated below.
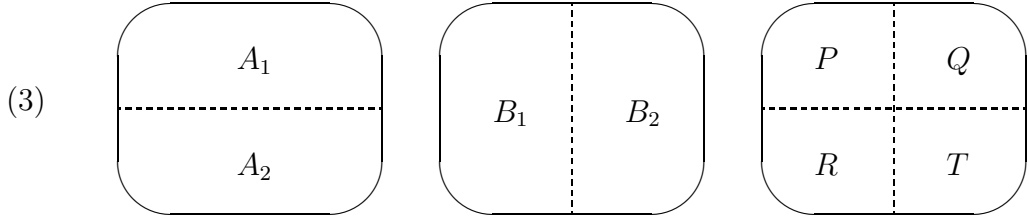
\begin{equation}\label{e:4sets}
\begin{gathered}
\begin{picture}(100,90)(20,0)
  \put(50,40){\oval(100,80)}
  \put(0,40){\dashbox{2}(100,0){}}
  \put(45,55){$A_1$}
  \put(45,15){$A_2$}
\end{picture}
\quad
\begin{picture}(100,90)(10,0)
  \put(50,40){\oval(100,80)}
  \put(50,0){\dashbox{2}(0,80){}}
  \put(20,35){$B_1$}
  \put(70,35){$B_2$}
\end{picture}
\quad
\begin{picture}(100,90)(0,0)
\put(50,40){\oval(100,80)}
  \put(0,40){\dashbox{2}(100,0){}}
  \put(50,0){\dashbox{2}(0,80){}}
  \put(20,55){$P$}
  \put(70,55){$Q$}
  \put(20,15){$R$}
  \put(70,15){$T$}
\end{picture}
\end{gathered}
\end{equation}

For any such pair of set compositions, the diagram
\begin{equation}\label{e:comp}
\begin{gathered}
\xymatrix@R+2pc@C-5pt{
\bH[P] \otimes \bH[Q] \otimes \bH[R] \otimes \bH[T] \ar[rr]^{\cong} & &
\bH[P] \otimes \bH[R] \otimes \bH[Q] \otimes \bH[T] \ar[d]^{\mu_{P,R}
\otimes \mu_{Q,T}}\\
\bH[A_1] \otimes \bH[A_2] \ar[r]_-{\mu_{A_1,A_2}}\ar[u]^{\Delta_{P,Q} \otimes
\Delta_{R,T}} & \bH[I] \ar[r]_-{\Delta_{B_1,B_2}} & \bH[B_1] \otimes
\bH[B_2]
}
\end{gathered}
\end{equation}
must commute. The top arrow stands for the map that interchanges the middle factors.
\end{compatibility}

In addition, the Hopf monoid $\bH$ is \emph{connected} if
 $\bH[\emptyset]=\field$ and the maps
\[
\xymatrix{
\bH[I]\otimes\bH[\emptyset]  \ar@<0.5ex>[r]^-{\mu_{I,\emptyset}} & 
\bH[I] \ar@<0.5ex>[l]^-{\Delta_{I,\emptyset}} 
}
\qqand
\xymatrix{
\bH[\emptyset]\otimes\bH[I]  \ar@<0.5ex>[r]^-{\mu_{\emptyset,I}} & 
\bH[I] \ar@<0.5ex>[l]^-{\Delta_{\emptyset,I}} 
}
\]
are the canonical identifications.

The collection $\mu$ is the \emph{product} and the collection $\Delta$
is the \emph{coproduct} of the Hopf monoid $\bH$. For any Hopf monoid $\bH$ the existence
of the antipode map $S\colon \bH\to\bH$ is guaranteed and it can be computed using Takeuchi's formula as follows. For any finite set $I$
\begin{equation}\label{eq:takeuchi}
 S_I = \sum_{k=1}^{|I|}\sum_{(A_1,\ldots,A_k)\models I} (-1)^{k} \mu_{A_1,\ldots,A_k} \Delta_{A_1,\ldots,A_k}
 = \sum_{A\models I} (-1)^{\ell(A)} \mu_A \Delta_A\,.
\end{equation}
Here, for $k=1$, we have $\mu_{A_1}=\Delta_{A_1}=\Id_I$ the identity map on $\bH[I]$, and for $k>1$,
 $$  \mu_{A_1,\ldots,A_k} = \mu_{A_1,I\setminus A_1}(\Id_{A_1}\otimes\mu_{A_2,\ldots,A_k})\qquad\text{and}\qquad
       \Delta_{A_1,\ldots,A_k} =(\Id_{A_1}\otimes\Delta_{A_2,\ldots,A_k}) \Delta_{A_1,I\setminus A_1}.
 $$

A Hopf monoid is (co)commutative if the left (right) diagram below commutes
for all set compositions $(A_1,A_2)\models I$.
\begin{equation}\label{e:comm}
\begin{gathered}
\xymatrix@C-15pt{
\bH[A_1]\otimes\bH[A_2] \ar[rr]^-{\tau_{A_1,A_2}} \ar[rd]_{\mu_{A_1,A_2}}  & & \bH[A_2]\otimes\bH[A_1] \ar[ld]^{\mu_{A_2,A_1}} \\
& \bH[I] &
}
\qquad
\xymatrix@C-15pt{
\bH[A_1]\otimes\bH[A_2] \ar[rr]^-{\tau_{A_1,A_2}}   & &\bH[A_2]\otimes\bH[A_1]  \\
& \bH[I] \ar[lu]^{\Delta_{A_1,A_2}} \ar[ru]_{\Delta_{A_2,A_1}} &
}
\end{gathered}
\end{equation}
The arrow $\tau_{A_1,A_2}$ stands for
the map that interchanges the factors.

A morphism of Hopf monoids $f:\bH\to{\bf Q}$ is a morphism of species that
commutes with $\mu$ and $\Delta$.

\subsection{The Hopf monoid of linear orders $\bL$}\label{ss:order}

For any finite set $I$ let
$\rL[I]$ be the set of all linear orders on $I$. For instance, if $I=\{a,b,c\}$,
\[
\rL[I]=\{abc,\,bac,\,acb,\,bca,\,cab,\,cba\}.
\]
The vector species $\bL$ is such that $\bL[I]:=\field\rL[I]$ is the vector space with basis $\rL[I]$. 

Let $(A_1,A_2)\models I$.
Given linear orders $\alpha_1,\alpha_2$ on $A_1,A_2$, respectively,
their concatenation $\alpha_1\cdot \alpha_2$ 
is a linear order on $I$. This is the linear order given by $\alpha_1$ followed by $\alpha_2$.
Given a linear order $\alpha$ on $I$ and $P\subseteq I$, the restriction 
 $\alpha|_P$ is the ordering in $P$ given by the order $\alpha$.
These operations give rise to maps
\begin{equation}\label{e:L}
\begin{aligned}
\rL[A_1]\times\rL[A_2] & \to \rL[I] &\qquad\qquad \rL[I]  & \to \rL[A_1]\times\rL[A_2]\\
(\alpha_1,\alpha_2) & \to \alpha_1\cdot\alpha_2
& \alpha  & \to (\alpha |_{A_1},\alpha |_{A_2}).
\end{aligned}
\end{equation}
Extending by linearity, we obtain linear maps
\[
 \mu_{A_1,A_2}:\bL[A_1]\otimes\bL[A_2] \to \bL[I]
 \qand
\Delta_{A_1,A_2}:\bL[I]\to \bL[A_1]\otimes\bL[A_2]
\]
which turn $\bL$ into a cocommutative but not commutative Hopf monoid.

\subsection{The Hopf monoid of set partitions $\bPi$}\label{ss:partition}

A \emph{partition} of a finite set $I$ is a collection $X$ of disjoint nonempty subsets whose union is $I$. The subsets are the \emph{blocks} of $X$.

Given a partition $X$ of $I$ and $P\subseteq I$, the restriction
$X|_P$ is the partition of $P$ whose blocks are 
the nonempty intersections of the blocks of $X$ with $P$.
Given $(A_1,A_2)\models I$.
and partitions $X_i$ of $A_i$, $i=1,2$, 
the union $X_1\cup X_2$ is the partition of $I$ 
whose blocks are the blocks of $X_1$ and the blocks of $X_2$.

Let ${\mathbf \pi}[I]$ denote the set of partitions of $I$ and $\bPi[I]=\field {\mathbf \pi}[I]$ the vector
space with basis ${\mathbf \pi}[I]$. 
A Hopf monoid structure on $\bPi$ is defined and studied in~\cite{AM:2010,BakerJarvisBergeronThiem,AguiarBergeronThiem,BergeronZabrocki}. Among its various linear bases, we are interested in the {\sl power-sum} basis 
on which the operations are as follows. 
The product
\[
\mu_{A_1,A_2}: \bPi[A_1] \otimes \bPi[A_2] \to \bPi[I] 
\]
is given by
\begin{equation}\label{e:prod-m}
\mu_{A_1,A_2}({X_1} \otimes {X_2}) = X_1\cup X_2.
\end{equation}
for $X_i\in {\mathbf \pi}[A_i]$ and extended linearly. The coproduct
\[
\Delta_{A_1,A_2}: \bPi[I]  \to \bPi[A_1]\otimes\bPi[A_2]
\]
is given by
\begin{equation}\label{e:coprod-m}
\Delta_{A_1,A_2}(X) = 
\begin{cases}
{X|_{A_1}} \otimes {X|_{A_2}}  & \text{if $A_1$ is the union of some blocks of $X$,} \\
0 & \text{otherwise,}
\end{cases}
\end{equation}
for $X\in {\mathbf \pi}[I]$ and extended linearly.
These operations turn the species $\bPi$ into a Hopf monoid that is both commutative and cocommutative.

\subsection{The Hopf monoid of simple graphs $\bf G$}\label{ss:graph}
A \emph{(simple) graph} $g$ on a finite set $I$ is a collection $E$ of subsets of $I$ of size 2.
The elements of $I$ are the \emph{vertices} of $g$.
There is an \emph{edge} between two vertices $i, j$ if $\{i,j\}\in E$.

 Given a graph $g$ on $I$ and $P\subseteq I$, the restriction
$g|_P$ is the graph on the vertex set $P$ whose edges are the edges of $g$ between elements of $P$.
Let $(A_1, A_2)\models I$.
Given graphs $g_i$ of $A_i$, $i=1,2$, 
their union is the graph $g_1\cup g_2$ of $I$ 
whose edges are those of $g_1$ and those of $g_2$.

Let $\rG[I]$ denote the set of graphs on $I$ and $\bG[I]=\field \rG[I]$ the vector
space with basis $\rG[I]$. 
A Hopf monoid structure on $\bG$ is defined from
\begin{equation}\label{e:G}
\begin{aligned}
\rG[A_1]\times\rG[A_2] & \to \rG[I] &\qquad\qquad \rG[I]  & \to \rG[A_1]\times\rG[A_2]\\
(g_1,g_2) & \to g_1\cup g_2
& g  & \to (g |_{A_1},g |_{A_2}).
\end{aligned}
\end{equation}
Extending by linearity, we obtain linear maps
\[
 \mu_{A_1,A_2}:\bG[A_1]\otimes\bG[A_2] \to \bG[I]
 \qand
\Delta_{A_1,A_2}:\bG[I]\to \bG[A_1]\otimes\bG[A_2]
\]
These operations turn the species $\bG$ into a Hopf monoid that is both commutative and cocommutative.

\subsection{The Hopf monoid of simple hypergraphs  $\bf HG$}\label{ss:hypergraph}

Let  $2^I$ denote the collection of subsets of $I$. Let $\bf HG[I]=\field {\bf hg}[I]$ be the space spanned by the basis $\bf hg[I]$ where
      $${\bf hg}[I]=\big\{h \subseteq 2^I\,:  U\in h\text{ implies }|U|\ge 2\big\}$$
An element $h\in{\bf hg}[I]$ is a \emph{hypergraph on $I$}.
For $(P,T)\models I$ and $h,k\in {\bf hg}[I]$, the multiplication is given by $\mu_{P,T}(h,k)=h\cup k$ and the comultiplication 
is given by $\Delta_{P,T}(h) = h |_P\otimes h |_T$ where $h |_P=\{U\in h : U\cap P = U\}$. Extending these definition linearly we have that $\bf HG$ is 
commutative and cocommutative Hopf monoid. 

%

\subsection{The Hadamard product}\label{ss:hadamard}

Given two species $\bH$ and $\bf Q$, their \emph{Hadamard product} is the species $\bH\times\bf Q$ defined by
\[
(\bH\times{\bf Q})[I] = \bH[I]\otimes{\bf Q}[I].
\]
If $\bH$ and $\bf Q$ are Hopf monoids, then so is $\bH\times\bf Q$, with the following operations. Let $(A_1, A_2)\models I$. The product is
\[
\xymatrix@C+12pt{
(\bH\times{\bf Q})[A_1]\otimes(\bH\times{\bf Q})[A_2] \ar@{=}[d] \ar@{-->}[rr] &&
 (\bH\times{\bf Q})[I] \ar@{=}[dd] \\
\bH[A_1]\otimes {\bf Q}[A_1]\otimes \bH[A_2]\otimes {\bf Q}[A_2]
\ar[d]_{\cong} && \\
\bH[A_1]\otimes \bH[A_2]\otimes {\bf Q}[A_1]\otimes {\bf Q}[A_2]
\ar[rr]_-{\mu_{A_1,A_2}^{\bf H}\otimes\mu_{A_1,A_2}^{\bf Q}} &&
\bH[I]\otimes{\bf Q}[I]
}
\]
The coproduct is defined similarly. If $\bH$ and ${\bf Q}$ are (co)commutative, then so is $\bH\times{\bf Q}$. 

\subsection{Linearized Hopf monoids}\label{ss:stronglin} A \emph{set species} $\bh$ is a collection of sets $\bh[I]$,
one for each finite set $I$, equivariant with respect to bijections $I\cong J$.
We say that $\bh$ is a basis for a Hopf monoid $\bH$ if for every finite set $I$ we have that $\bH[I]=\field \bh[I]$, the vector space with basis $\bh[I]$. We say that the monoid $\bH$ is \emph{linearized} in the basis $\bh$ if the product and coproduct maps have the following properties.
The product
\[
\mu_{A_1,A_2}: \bH[A_1] \otimes \bH[A_2] \to \bH[I] 
\]
is the linearization of a map
\begin{equation}\label{e:prod-lin}
\mu_{A_1,A_2}: \bh[A_1] \otimes \bh[A_2] \to \bh[I] 
\end{equation}
and the coproduct
\[
\Delta_{A_1,A_2}: \bH[I]  \to \bH[A_1]\otimes\bH[A_2]
\]
is the linearization of a map
\begin{equation}\label{e:coprod-lin}
\Delta_{A_1,A_2}: \bh[I]  \to (\bh[A_1]\otimes\bh[A_2])\cup\{0\}\,.
\end{equation}
From now on, we will use capital letters for vector species and lower case for set species.

The Hopf monoids $\bL$, $\bPi$, $\bG$ and $\bf HG$ are linearized in the bases $\bf l$, $\pi$, $\bf g$ and $\bf hg$ respectively. As remarked in \cite{MarbergLin}, many of the Hopf monoids in the literature are linearized in some basis. On the other hand, the Hopf monoid $\bL^{\star}$ is not linearized in $\bf l^{*}$, where $\bL^{\star}$ denotes the Hopf monoid \emph{dual} to $\bL$ (see \cite{AM:2010}).

\subsection{Functors $\Kc$ and $\Kcb$} \label{ss:K_Kbar} 
As describe in ~\cite{AM:2010}, there are some interesting functors from the category of species to the category of graded vector spaces. Let $[n]:=\{1,2,\ldots,n\}$ and assume throughout that $\text{char}(\field)=0$. Given a species $\bH$, the symmetric group $S_n$ acts on $\bH[n]$ by relabelling. Define the functors $\Kc$ and $\Kcb$ by
  $$ \Kc({\bf H}) = \bigoplus_{n\ge 0} {\bf H}[n]
  \qquad\qquad
    \Kcb({\bf H}) = \bigoplus_{n\ge 0} {\bf H}[n]_{S_n}
  $$
where
  $${\bf H}[n]_{S_n}={\bf H}[n]\Big/ \langle x-{\bf H}[\sigma](x) \mid \ \sigma\in S_n; \ x\in {\bf H}[n]\rangle$$
 denotes the quotient space of equivalence classes under the $S_n$ action. When $\bH$ is a Hopf monoid, 
  we can build a product and coproduct
on $ \Kc({\bf H})$ and $ \Kcb({\bf H})$ from those of $\bH$ together
with certain canonical transformations.
For example, one has that
\[
\Kcb(\bL)=\field[x]
\]
is the polynomial algebra on one generator, while $\Kc(\bL)$
is the Hopf algebra introduced by Patras and Reutenauer in~\cite{PR:2004}. In this case, the antipode map $S:\bL\rightarrow\bL$ is such that
for $\alpha=a_1\cdots a_n\in[n]$
   $$ S_{[n]}(\alpha)=(-1)^n a_n\cdots a_1.$$
  However, the antipode of the graded Hopf algebra $\Kc[\bL]$  is not given by the formula above (see Section~\ref{s:KLxH}). On the other hand, in the Hopf algebra $\Kcb[\bL]\cong \field[t]$, the  antipode is given by $S(t^n)=(-1)^nt^n$ and is functorial image of the map above. This is not accident: the functor $\Kc$ may not preserve the antipode but the functor $\Kcb$ always does.


A very interesting relation between the functors $\Kc$ and $\Kcb$ is given in~\cite[Theorem~15.13]{AM:2010} as follows
\begin{equation}\label{eq:KbarK}
 \Kcb(\bL\times\bH) \cong \Kc(\bH).
\end{equation}
where $\bH$ is an arbitrary Hopf monoid. In this paper we aim to make use of this relation to study the antipode problem for some Hopf algebras.

\section{Antipode for linearized Hopf Monoid $\bL\times \bH$}\label{s:antiLxH}

In this section we show a multiplicity free and cancelation free formula for the antipode of Hopf monoids of the form $\bL\times \bH$ where $\bH$ is linearized in some basis. Thus, by (\ref{eq:KbarK}), we obtain an antipode formula for $\Kc(\bH)$ as well. However, in $\Kc(\bH)$ this antipode formula may not be cancellation free. 

\subsection{Antipode Formula for $\bL\times \bH$}\label{ss:UsingTakeushi}

Let $\bH$ be a Hopf monoid linearized in the basis $\bh$. We intend to resolve the cancelations in the Takeuchi formula for $\bL\times\bH$. For a fixed finite set $I$ let $(\alpha,x)\in ({\bf l}\times\bh)[I]$, that is, $\alpha$ is a linear ordering on $I$ and $x$ is an element of $\bh [I]$. From~\eqref{eq:takeuchi} we have
\begin{equation}\label{eq:takeuchi1}
 S_I(\alpha,x) =  \sum_{A\models I} (-1)^{\ell(A)} \mu_A \Delta_A(\alpha,x)
  = \sum_{A\models I \atop \Delta_A(x)\ne 0} (-1)^{\ell(A)} (\alpha_A,x_A),
\end{equation}
 summing over all $A=(A_1,\ldots,A_k)\models I$, where $\alpha_A$ denotes the element in ${\bf l} [I]$ given 
  $$\alpha_A= \alpha|_{A_1}\cdot  \alpha|_{A_2}\cdots \alpha_ \alpha|_{A_k}
$$ and if $\Delta_A(x)\ne 0$, then
  $$x_A=\mu_A\Delta_A(x) \in \bh[I].$$
Each composition $A$ gives rise to single elements $\alpha_A$ and $x_A$ since $\bL$ and $\bH$ are linearized in the basis $\bf l$ and $\bh$, respectively. 
We can thus rewrite equation~\eqref{eq:takeuchi1} as
\begin{equation}\label{eq:takeuchi2}
 S_I(\alpha,x) = \sum_{(\beta,y)\in ({\bf l}\times\bh)[I]}\left(\sum_{A\models I \atop (\alpha_A,x_A)=(\beta,y)} (-1)^{\ell(A)}\right) (\beta,y).
\end{equation}
Let 
 $${\mathcal C}_{\alpha,x}^{\beta,y}=\big\{A\models I : (\alpha_A,x_A)=(\beta,y)\big\}$$

The following theorem provides us a multiplicity-free and cancellation-free formula for the antipode of $\bL\times\bH$. Using the above notation we have
\begin{theorem}\label{thm:antiLxP} Let $\bH$ be a linearized Hopf monoid in the basis $\bf h$. For $(\alpha, x)\in({\bf l}\times\bh)[I]$ we obtain
 \begin{equation}\label{eq:thm1}
 S_I(\alpha,x) = \sum_{(\beta,y)\in ({\bf l}\times\bh)[I]} c_{\alpha,x}^{\beta,y} \,(\beta,y),
\end{equation}
where 
  $$c_{\alpha,x}^{\beta,y}=\sum_{A\in {\mathcal C}_{\alpha,x}^{\beta,y}} (-1)^{\ell(A)}=\pm1 \text{ or } 0.$$
\end{theorem}

The proof of this theorem will be given in Section~\ref{ss:proof1}. We make use of the refinement order on set compositions to show that the set ${\mathcal C}_{\alpha,x}^{\beta,y}$ has a unique minimum. We will use this fact along with other properties to construct sign reversing involutions on ${\mathcal C}_{\alpha,x}^{\beta,y}$ and the result will follow once we understand the fixed points of such involutions.

\subsection{Minimal element of ${\mathcal C}_{\alpha,x}^{\beta,y}$}\label{ss:refinementonC}
Given $A=(A_1,\ldots,A_k)$ and $B=(B_1,\ldots,B_\ell)$ set compositions on a set $I$, 
we say that $A$ \emph{refines} $B$, and we write $A\le B$, if the parts of $B$ are union of consecutive parts of $A$. For example
 $$A = \big(\{1,4\},\{2\},\{5,7\},\{3\},\{9\},\{6,8\}\big) \le \big(\{1,4\},\{2,3,5,7\},\{6,8,9\}\big)=B$$
but $A$ does not refine $\big(\{1,4,5,7\},\{2\},\{3\},\{9\},\{6,8\}\big)$. Denote by $(\mathcal P_I,\leq)$ the poset of set compositions of $I$, ordered by refinement.
In what follows we will write $(14,2,57,3,9,68)$ instead of $\big(\{1,4\},\{2\},\{5,7\},\{3\},\{9\},\{6,8\}\big)$. Consider the order $\le$ restricted to the set ${\mathcal C}_{\alpha,x}^{\beta,y}$.

\begin{lemma}\label{le:min}
If ${\mathcal C}_{\alpha,x}^{\beta,y}\ne\emptyset$, then there is a unique minimal element in $({\mathcal C}_{\alpha,x}^{\beta,y},\le)$.
\end{lemma}
\begin{proof} Suppose that  $A=(A_1,\ldots,A_k)$ and $B=(B_1,\ldots,B_\ell)$  are minimal in $\in {\mathcal C}_{\alpha,x}^{\beta,y}$ and $A\neq B$. We have that $\alpha_A=\alpha_B=\beta$ and $x_A=x_B=y$. Since $\alpha_A=\beta$ then the parts of $A$ appear consecutively in  $\beta$ and the same is true for the parts of $B$. For example if $\alpha=abcdef$ and $\beta=bcfade$, then for $A=(bc,f,ad,e)$ and $B=(bc,f,a,de)$ we have $\alpha_A=\alpha_B=\beta$.

 Let $1\le i\le k$ be the smallest index such that $A_i\ne B_i$, and assume without loss of generality that $|A_i|>|B_i|$. If $i=k$ then $B$ refines $A$ and this is a contradiction. Hence we assume that $i<k$ and  we now build a composition $C$ that refines $A$ such that $C\in {\mathcal C}_{\alpha,x}^{\beta,y}$, which will contradict again the minimality of $A$. Since $\alpha_A=\alpha_B$ our choice of $i$ implies that $B_i\subset A_i$.
Let $U=A_i\setminus B_i$. We claim that for $C=(A_1,\ldots,A_{i-1},B_i,U,A_{i+1},...,A_k)$
\begin{enumerate}
\item[(a)] $C<A$
\item[(b)] $\alpha_C=\alpha_A=\beta$
\item[(c)] $x_C=x_A=y$
\end{enumerate}
The items (a) and (b) are straightforward. Now we proceed to show (c) by showing that
\begin{equation}\label{eq:xCisxA}
x_A=\mu_{P,T}\Delta_{P,T}\mu_A\Delta_A(x)=\mu_C\Delta_C(x).
\end{equation}

Let $P=B_1\cup\cdots\cup B_i$ and $T=B_{i+1}\cup\cdots\cup B_\ell$.
We claim that
\begin{equation}\label{eq:Tcutx}
  x_B=\mu_{P,T}\Delta_{P,T}(x_B).
\end{equation}
To see this, we use associativity of $\mu$ to write $\mu_B=\mu_{P,T}(\mu_{B_1,\ldots,B_i}\otimes\mu_{B_{i+1},\ldots,B_\ell})$. Also, the compatibility relation~\eqref{e:comp} with $Q=R=\emptyset$ gives
  $$\Delta_{P,T}\mu_{P,T}=\Id_P\otimes \Id_T.$$
Hence equation~\eqref{eq:Tcutx} follows from
  $$
   \begin{aligned}
   \mu_{P,T}\Delta_{P,T}(x_B)&=\mu_{P,T}\Delta_{P,T}\mu_{P,T}(\mu_{B_1,\ldots,B_i}\otimes\mu_{B_{i+1},\ldots,B_\ell})\Delta_B(x)\\
      &=\mu_{P,T}(\mu_{B_1,\ldots,B_i}\otimes\mu_{B_{i+1},\ldots,B_\ell})\Delta_B(x)=\mu_B\Delta_B(x)=x_B.
   \end{aligned} 
  $$
Using again~\eqref{eq:Tcutx} and the fact that $x_A=x_B$ we show the first equality in \eqref{eq:xCisxA}:
\begin{equation}\label{eq:TcutS}
  x_A=\mu_{P,T}\Delta_{P,T}(x_A)=\mu_{P,T}\Delta_{P,T}\mu_A\Delta_A(x).
\end{equation}
 Now let $P'=A_1\cup\cdots\cup A_{i-1}$, $Q'=\emptyset$, $R'=B_i $ and $T'=T$ in the compatibility relation~\eqref{e:comp}:
  $$
  \begin{aligned}
  \Delta_{P,T}\mu_A &= \Delta_{P'\cup R',T'}\mu_{P',R'\cup T'}(\mu_{A_1,\ldots,A_{i-1}}\otimes \mu_{A_i,\ldots,A_k})\\
                                 &=(\mu_{P',R'}\otimes \Id_{T'})(\Id_{P'}\otimes\Delta_{R',T'})(\mu_{A_1,\ldots,A_{i-1}}\otimes \mu_{A_i,\ldots,A_k})\\
                                 &=(\mu_{P',R'}\otimes \Id_{T'})(\mu_{A_1,\ldots,A_{i-1}}\otimes\Id_{R'}\otimes\Id_{T'})
                                               (\Id_{A_{1}} \otimes\cdots\otimes \Id_{A_{i-1}}\otimes\Delta_{R',T'}\mu_{A_i,\ldots,A_k})\\
  \end{aligned}
  $$
  We now expand $\Delta_{R',T'}\mu_{A_i,\ldots,A_k}$ using similar manipulations. Let $T''=A_{i+1}\cup\cdots\cup A_k$,
    $$
  \begin{aligned}
  \Delta_{R',T'}\mu_{A_i,\ldots,A_k} &= \Delta_{B_i , U\cup T''}\mu_{B_i\cup U,T''}(\Id_{A_i}\otimes \mu_{A_{i+1},\ldots,A_k})\\
                                 &=(\Id_{B_i}\otimes \mu_{U,T''})(\Delta_{B_i,U}\otimes\Id_{T''})(\Id_{A_i}\otimes \mu_{A_{i+1},\ldots,A_k})\\
                                 &=(\Id_{B_i}\otimes \mu_{U,T''})(\Id_{B_i}\otimes \Id_U\otimes \mu_{A_{i+1},\ldots,A_k})(\Delta_{B_i,U}\otimes\Id_{A_{i+1}}
                                        \otimes\cdots\otimes \Id_{A_k})\\
                                 &=(\Id_{B_i}\otimes \mu_{U,A_{i+1},\ldots,A_k})(\Delta_{B_i,U}\otimes\Id_{A_{i+1}}  \otimes\cdots\otimes \Id_{A_k})\\
  \end{aligned}
  $$
Remark that since $R'=B_i$,
  $$
  \begin{aligned}
  \mu_C&=\mu_{P,T}(\mu_{A_1,\ldots,A_{i-1},B_i}\otimes \Id_{T'})(\Id_{A_{1}} \otimes\cdots\otimes \Id_{A_{i-1}}\otimes\Id_{B_i}\otimes \mu_{U,A_{i+1},\ldots,A_k})\\
  &=\mu_{P,T}(\mu_{P',R'}\otimes \Id_{T'})(\mu_{A_1,\ldots,A_{i-1}}\otimes\Id_{R'}\otimes\Id_{T'})(\Id_{A_{1}} \otimes\cdots\otimes \Id_{A_{i-1}}\otimes\Id_{B_i}\otimes \mu_{U,A_{i+1},\ldots,A_k})\\
  \end{aligned}
  $$
  and 
  $$\Delta_C=(\Id_{A_{1}} \otimes\cdots\otimes \Id_{A_{i-1}}\otimes\Delta_{B_i,U}\otimes\Id_{A_{i+1}}  \otimes\cdots\otimes \Id_{A_k})\Delta_A.$$
 Making use of the expressions given above for  $\Delta_{P,T}\mu_A$, and comparing with $\mu_C\Delta_C$ we get
  $$x_A=\mu_{P,T}(\Delta_{P,T}\mu_A)\Delta_A(x)=\mu_C\Delta_C(x)=x_C.$$
We conclude that the composition $C$ satisfies (a), (b) and (c) contradicting the choice of $A$, hence we must have a unique minimal element in $({\mathcal C}_{\alpha,x}^{\beta,y},\le)$.
\end{proof}

For the rest of this section, let $\alpha,\beta,x$ and $y$ be fixed and let $\Lambda=(\Lambda_1,\Lambda_2,\ldots,\Lambda_m)$ be the minimum of ${\mathcal C}_{\alpha,x}^{\beta,y}\ne \emptyset$.
For any $A\in{\mathcal C}_{\alpha,x}^{\beta,y}$ let $[\Lambda,A]$ denote the interval  $\big\{ B\models I : \Lambda\le B\le A\big\}\subseteq \mathcal P_I$. A priori, this interval does not need to be contained in ${\mathcal C}_{\alpha,x}^{\beta,y}$, but the following lemma shows that this is indeed the case.

\begin{lemma}\label{le:lowerideal}
If ${\mathcal C}_{\alpha,x}^{\beta,y}\ne\emptyset$, then for any $A\in{\mathcal C}_{\alpha,x}^{\beta,y}$ we have that 
$[\Lambda,A]\subseteq {\mathcal C}_{\alpha,x}^{\beta,y}$.
\end{lemma}

\begin{proof} Let $A=(A_1,A_2,\ldots,A_k)$. From Lemma~\ref{le:min} we know that $\Lambda\le A$. We proceed by induction on $r = \ell(\Lambda)-\ell(A)$.
If $r=0$, then we have that $A=\Lambda$ and the result follows.
Suppose the result holds for $r>0$ and let $A$ be such that $r+1= \ell(\Lambda)-\ell(A)$. Let $B\in \mathcal P_I$ such that $\Lambda\leq B< A$ with $\ell(B)-\ell(A)=1$. Hence there is a unique $i$ such that  $A=(B_1,\ldots,B_i\cup B_{i+1},B_{i+2},\ldots,B_{k+1})$. We aim to show that $B\in{\mathcal C}_{\alpha,x}^{\beta,y}$, and then by induction hypothesis $[\Lambda,B]\subseteq {\mathcal C}_{\alpha,x}^{\beta,y}$. 

Since $\Lambda\le B$, there is a unique $j$ such that 
 $$ B_1\cup\ldots\cup B_i=\Lambda_1\cup\cdots\cup \Lambda_j.$$
Let $P=\Lambda_1\cup\cdots\cup \Lambda_j$ and $Q=\Lambda_{j+1}\cup\cdots\cup \Lambda_m$. Arguing as in equations~\eqref{eq:Tcutx} and~\eqref{eq:TcutS} we have that 
  $$
  \begin{aligned}
   y&=\mu_\Lambda\Delta_\Lambda(x)=\mu_{P,Q}(\mu_{\Lambda_1,\ldots, \Lambda_j}\otimes \mu_{\Lambda_{j+1},\ldots, \Lambda_m})\Delta_\Lambda(x)\\
    &=\mu_{P,Q}\Delta_{P,Q}\mu_{P,Q}(\mu_{\Lambda_1,\ldots, \Lambda_j}\otimes \mu_{\Lambda_{j+1},\ldots, \Lambda_m})\Delta_\Lambda(x)
        =\mu_{P,Q}\Delta_{P,Q}(y)\\
    &=\mu_{P,Q}\Delta_{P,Q}\mu_A\Delta_A(x)= \mu_B\Delta_B(x)=x_B
   \end{aligned}
  $$
  The same argument shows that $\beta=\mu_{P,Q}\Delta_{P,Q}\mu_A\Delta_A(\alpha)= \mu_B\Delta_B(\alpha)=\alpha_B$. Hence $B\in {\mathcal C}_{\alpha,x}^{\beta,y}$. We can now appeal to the induction hypothesis and conclude that for each such $B$ the interval $[\Lambda,B]\subseteq {\mathcal C}_{\alpha,x}^{\beta,y}$ and thus the claim follows. Moreover, we conclude that ${\mathcal C}_{\alpha,x}^{\beta,y}$ is a lower ideal of the subposet
  $[\Lambda,(I)]=\{B\models I : \Lambda \le B\}$.
\end{proof}

\begin{lemma}\label{le:upperideal} For $ {\mathcal C}_{\alpha,x}^{\beta,y}\ne \emptyset$,
the minimal elements of $[\Lambda,(I)]\setminus {\mathcal C}_{\alpha,x}^{\beta,y}$ are all of the form
 $$(\Lambda_1,\ldots,\Lambda_{i-1},\Lambda_i\cup\Lambda_{i+1}\cup\cdots\cup \Lambda_j,\Lambda_{j+1},\ldots,\Lambda_m)$$
for some $1\le i<j\le m$.
\end{lemma}
\begin{proof}
For $ {\mathcal C}_{\alpha,x}^{\beta,y}\ne\emptyset$, let $B\in [\Lambda,(I)]$ be minimal such that $B\notin  {\mathcal C}_{\alpha,x}^{\beta,y}$.
That is $\alpha_B\ne \beta$ or $x_B\ne y$.  Let us first consider the case when $\alpha_B\ne \beta$. 
If $\alpha_B\ne \beta$ we must have at least one part of $B$ that contains $\Lambda_i\cup\Lambda_{i+1}$ where the largest entry of $\alpha|_{\Lambda_i}$, say $a$, is such that $a>_{\alpha}b$, where $b$ is the smallest entry of $\alpha|_{\Lambda_{i+1}}$. Hence, 
 $$\Lambda<B\le (\Lambda_1,\ldots,\Lambda_{i-1},\Lambda_i\cup\Lambda_{i+1},\Lambda_{i+2},\ldots,\Lambda_m),$$
hence $B= (\Lambda_1,\ldots,\Lambda_{i-1},\Lambda_i\cup\Lambda_{i+1},\Lambda_{i+2},\ldots,\Lambda_m)$. Thus the claim follows when $\alpha_B\neq\beta$.

We now consider the case $x_B\ne y$.
Assume that $B=(B_1,...,B_k)$ has at least two parts that are union of consecutive parts of $\Lambda$. Each part $B_s$ of $B$ is of the form $\Lambda_{a_s}\cup\cdots\cup\Lambda_{b_s}$, where $1\leq a_s\leq b_s\leq m$. For each $1\leq s\leq k$ consider the composition 
  $$ C_{(s)}=(\Lambda_1,\ldots,\Lambda_{a_s-1},B_s,\Lambda_{b_s+1},\ldots,\Lambda_m).
  $$
It follows that $C_{(s)}$ refines $B$ (strictly) as there are at least two parts in $B$ that are union of consecutive parts of $\Lambda$. Hence $C_{(s)}\in {\mathcal C}_{\alpha,x}^{\beta,y}$ by the minimality of $B$, and thus $x_{C_{(s)}}=y$ for all $1\le s\le k$. Hence,
  $$
  \begin{aligned}
 & x|_{\Lambda_1}\otimes\ldots\otimes x|_{\Lambda_{a_s-1}}\otimes x|_{B_s}\otimes x|_{\Lambda_{b_s+1}} 
         \otimes\ldots\otimes x|_{\Lambda_m} \\
         &= \Delta_{C_{(s)}}(x_{C_{(s)}}) = \Delta_{C_{(s)}}(y) \\
         &= x|_{\Lambda_1}\otimes\ldots\otimes x|_{\Lambda_{a_s-1}}\otimes(x|_{\Lambda_{a_s}}\cdots x|_{\Lambda_{b_s}})\otimes x|_{\Lambda_{b_s+1}}
         \otimes\ldots\otimes x|_{\Lambda_m}
  \end{aligned}
  $$
  which gives us
   $$ x|_{B_s} = x|_{\Lambda_{a_s}}\cdots x|_{\Lambda_{b_s}}$$
   for all $1\le s\le k$. But this implies that
  $$ x_B = x|_{B_1} \cdots x|_{B_k} =(x_{\Lambda_{a_1}}\cdots x_{\Lambda_{b_1}})\cdots (x_{\Lambda_{a_k}}\cdots x_{\Lambda_{b_k}})=x_\Lambda=y,$$
  a contradiction. Hence there is no more than one part of $B$ that is not a single part of $\Lambda$.
\end{proof}

\subsection{First Sign Reversing Involution on $c_{\alpha,x}^{\beta,y}$}\label{ss:involution}
Throughout this section, let $\alpha,\beta,x$ and $y$ be fixed. If ${\mathcal C}_{\alpha,x}^{\beta,y}\ne\emptyset$, then we know that the subposet $({\mathcal C}_{\alpha,x}^{\beta,y},\le)$ is a lower ideal  with a unique minimum $\Lambda=(\Lambda_1,\Lambda_2,\ldots,\Lambda_m)$. 
We  define a signed reversing involution on the set ${\mathcal C}_{\alpha,x}^{\beta,y}$ that will cancel most of the terms in
\begin{equation}\label{eq:anticoef}
  c_{\alpha,x}^{\beta,y}=\sum_{A\in {\mathcal C}_{\alpha,x}^{\beta,y}} (-1)^{\ell(A)}.
\end{equation}

Using lemma~\ref{le:upperideal} we define an oriented graph $G_{\alpha,x}^{\beta,y}$ on the vertex set $[m]$ as follows. 
We have an (oriented) edge $(a,b)\in G_{\alpha,x}^{\beta,y}$, or simply $ab$, for each minimal element of $[\Lambda,(I)]\backslash {\mathcal C}_{\alpha,x}^{\beta,y}$.
More precisely, $ab$ is an edge in $ G_{\alpha,x}^{\beta,y}$ if the following holds:
\begin{enumerate}
 \item[$(1)$] $1\le a<b\le m$
 \item[$(2)$] For $B_{ab}=(\Lambda_1,\ldots,\Lambda_{a-1},\Lambda_a\cup\Lambda_{a+1}\cdots\cup \Lambda_b,\Lambda_{b+1},\ldots,\Lambda_m)$ we have
  $ \alpha_{B_{ab}}\ne \beta$ or $x_{B_{ab}}\ne y.$
 \item[$(3)$] For any $a<r<b$, we have $ \alpha_{B_{ar}}= \beta=\alpha_{B_{rb}}$ and $x_{B_{ar}}= y=x_{B_{rb}}.$
\end{enumerate}
Condition (2), in particular guarantees that no element $A\in{\mathcal C}_{\alpha,x}^{\beta,y}$ induces an edge in $G_{\alpha,x}^{\beta,y}$. Condition (3) allows us to conclude that the graph $G$ is \emph{non-nested}: for any pair of edges $(a,b),(c,d)\in G$ such that $a\le c\le b$, it follows  $a<c\le b<d$

\begin{example}\label{ex:LxHgraph}
Consider the Hopf monoid of graphs $\bG$ as in Section~\ref{ss:graph}.  $\bG$ is linearized in the basis $\bf g$.
Let $I= \{a,b,c,d,e,f,g,h\}$; $x,y\in{\bf g}[I]$ be the graphs
$$
x=
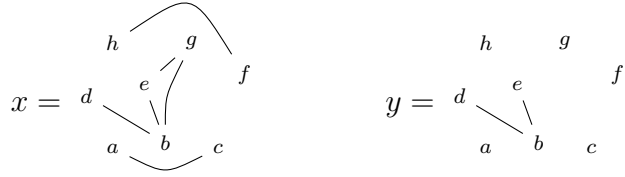
\begin{tikzpicture}[scale=.7,baseline=.5cm]
	\node (a) at (0,0) {$\scriptstyle a$};
	\node (b) at (1,.1) {$\scriptstyle b$};
	\node (c) at (2,0) {$\scriptstyle c$};
	\node (d) at (-.5,1) {$\scriptstyle d$};
	\node (e) at (.6,1.2) {$\scriptstyle e$};
	\node (h) at (0,2) {$\scriptstyle h$};
	\node (g) at (1.5,2) {$\scriptstyle g$};
	\node (f) at (2.5,1.4) {$\scriptstyle f$};
	\draw (a) .. controls (1,-.5) ..   (c); 
	\draw (b) .. controls (1,1) ..   (g); 
	\draw (h) .. controls (1.5,3) ..   (f); 
	\draw (g) --  (e); 
	\draw (b) --  (e); 
	\draw (d) --  (b); 
\end{tikzpicture} 
\qquad \qquad
y=
\begin{tikzpicture}[scale=.7,baseline=.5cm]
	\node (a) at (0,0) {$\scriptstyle a$};
	\node (b) at (1,.1) {$\scriptstyle b$};
	\node (c) at (2,0) {$\scriptstyle c$};
	\node (d) at (-.5,1) {$\scriptstyle d$};
	\node (e) at (.6,1.2) {$\scriptstyle e$};
	\node (h) at (0,2) {$\scriptstyle h$};
	\node (g) at (1.5,2) {$\scriptstyle g$};
	\node (f) at (2.5,1.4) {$\scriptstyle f$};
	\draw (b) --  (e); 
	\draw (d) --  (b); 
\end{tikzpicture} 
$$
and $\alpha,\beta$ be the orders $\alpha=abcdefgh$ and $\beta=abdefghc$. The minimum element of ${\mathcal C}_{\alpha,x}^{\beta,y}$
is $\Lambda=(a,bde,f,g,h,c)$ (notice that indeed $x_{\Lambda}=y$ and $\alpha_{\Lambda}=\beta$). Since $\Lambda$ has $6$ parts, the graph $G_{\alpha,x}^{\beta,y}$ is build on the set $[6]$.
We have 
$$G_{\alpha,x}^{\beta,y}=
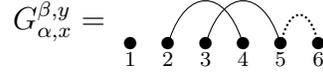
\begin{tikzpicture}[baseline=.2cm]
	\foreach \x in {1,2,3,4,5,6} 
		\node (\x) at (\x/2,0) [inner sep=-1pt] {$\bullet$};
	\node at (1/2,-.2) {$\scriptstyle 1$};
	\node at (2/2,-.2) {$\scriptstyle 2$};
	\node at (3/2,-.2) {$\scriptstyle 3$};
	\node at (4/2,-.2) {$\scriptstyle 4$};
	\node at (5/2,-.2) {$\scriptstyle 5$};
	\node at (6/2,-.2) {$\scriptstyle 6$};
	\draw (2) .. controls (2.5/2,.75) and (3.5/2,.75) .. (4); 
	\draw (3) .. controls (3.5/2,.75) and (4.5/2,.75) .. (5); 
	\draw  [densely dotted,thick]  (5) .. controls (5.25/2,.5) and (5.75/2,.5) .. (6); 
\end{tikzpicture} 
$$
In particular, notice that $12$ is not an edge as the element $B=(\Lambda_1\cup\Lambda_2,\Lambda_3,\dots,\Lambda_6)=(abde,f,g,h,c)\in {\mathcal C}_{\alpha,x}^{\beta,y}$.
The solid edges $(i,j)$ are drawn when the condition $ x_{B_{ij}}\ne y$ holds, the dotted edge $(5,6)$ indicates that  $ \alpha_{B_{56}}=abdefgch\ne \beta$. 
We now identify the set compositions in the interval $[\Lambda,(I)]$ with the set compositions of the interval  $[(1,2,\ldots,m),(12\cdots m)]$. 
We can represent ${\mathcal C}_{\alpha,x}^{\beta,y}$ as the following poset
$$
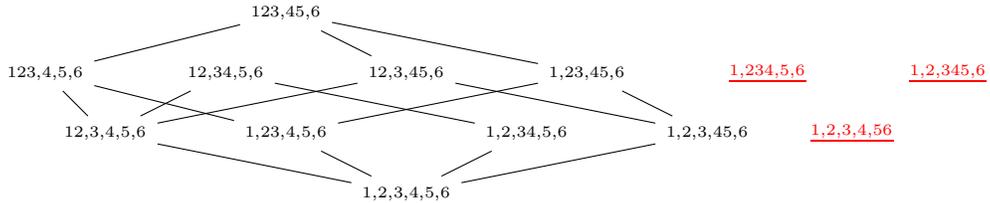
\begin{tikzpicture}[scale=1.6,baseline=.5cm]
	\node (0) at (0,0) {$\scriptscriptstyle 1,2,3,4,5,6$};
	\node (12) at (-2.5,.5) {$\scriptscriptstyle 12,3,4,5,6$};
	\node (23) at (-1,.5) {$\scriptscriptstyle 1,23,4,5,6$};
	\node (34) at (1,.5) {$\scriptscriptstyle 1,2,34,5,6$};
	\node (45) at (2.5,.5) {$\scriptscriptstyle 1,2,3,45,6$};
	\node (56) at (3.7,.5) {$\scriptscriptstyle \red{\underline{1,2,3,4,56}}$};
	\draw (0) -- (12);
	\draw (0) -- (23);
	\draw (0) -- (34);
	\draw (0) -- (45);
	\node (123) at (-3,1) {$\scriptscriptstyle 123,4,5,6$};
	\node (12_34) at (-1.5,1) {$\scriptscriptstyle 12,34,5,6$};
	\node (12_45) at (0,1) {$\scriptscriptstyle 12,3,45,6$};
	\node (234) at (3,1) {$\scriptscriptstyle \red{\underline{1,234,5,6}}$};
	\node (23_45) at (1.5,1) {$\scriptscriptstyle 1,23,45,6$};
	\node (345) at (4.5,1) {$\scriptscriptstyle \red{\underline{1,2,345,6}}$};
	\draw (12)--(123);
	\draw (12)--(12_34);
	\draw (12)--(12_45);
	\draw (23)--(123);
	\draw (23)--(23_45);
	\draw (34)--(12_34);
	\draw (45)--(23_45);
	\draw (45)--(12_45);
	\node (123_45) at (-1,1.5) {$\scriptscriptstyle 123,45,6$};
	\draw (23_45)-- (123_45);
	\draw (123)--(123_45);
	\draw (12_45)--(123_45);
\end{tikzpicture} 
$$
where the set compositions in red are the minimal compositions in $[\Lambda,(I)] \backslash {\mathcal C}_{\alpha,x}^{\beta,y}$, from Lemma~\ref{le:upperideal}.
\end{example}

\begin{remark}
As in the example above, for now on we will identify the set compositions in the interval $[\Lambda,(I)]$ with the set compositions of the interval  $[(1,2,\ldots,m),(12\cdots m)]$.
An element $A\in {\mathcal C}_{\alpha,x}^{\beta,y}$ is viewed as a set composition $A\models[m]$.
\end{remark}

For any set composition $B$ define its sign to be $sgn(B):=(-1)^{\ell (B)}$, where $\ell(B)$ is the length of $B$. We now define a sign reversing involution $\varphi\colon {\mathcal C}_{\alpha,x}^{\beta,y} \to {\mathcal C}_{\alpha,x}^{\beta,y}$, making use of auxiliary maps $\varphi _i$ for each $1\le i <m$, as follows.
Let $A=(A_1,\ldots,A_k)\in {\mathcal C}_{\alpha,x}^{\beta,y}$ and let $j$ be such that $i\in A_j$.

\medskip
\noindent {\bf $i$-Merge}: if $A_j=\{i\}$ and $(i,r)$ is not an edge of $G_{\alpha,x}^{\beta,y}$ for any $r\in A_{j+1}$, 
define
  $$ \varphi_i(A)=(A_1,\ldots,A_{j-1},\{i\}\cup A_{j+1},A_{j+2},\ldots, A_k).$$

\medskip
\noindent {\bf $i$-Split}: if $|A_j|>1$ and $i=\min(A_j)$ and  \big($j=1$ or $A_{j-1}\ne\{i-1\}$ or $(i-1,i)\in G_{\alpha,x}^{\beta,y}$\big), then
  $$ \varphi_i(A)=(A_1,\ldots,A_{j-1},\{i\},A_{j}\setminus\{i\},A_{j+1},\ldots, A_k).$$
  
 \medskip
\noindent {\bf $i$-Fix}: If we do not have an $i$-merge or an $i$-split, then
  $$\varphi_i(A)=A.$$
  
 Then the map $\varphi$ is defined as
 \begin{equation}\label{eq:involution}
   \varphi(A):=\begin{cases}
          A & \text{if $\varphi_i(A)=A$ for all $1\le i<m$,}\\
          &\\
          \varphi_{i_0}(A) & \text{for $i_0=\min\big\{i:\varphi_i(A)\ne A\big\}$, otherwise.}\\
       \end{cases}
 \end{equation}

\begin{lemma}\label{le:involution}
  $\varphi\colon {\mathcal C}_{\alpha,x}^{\beta,y} \to {\mathcal C}_{\alpha,x}^{\beta,y}$ is an involution.
\end{lemma}
 
\begin{proof} Let $A\in  {\mathcal C}_{\alpha,x}^{\beta,y}$. If $\varphi(A)=A$ the claim follows. Assume then that $\varphi(A)=A'\ne A$. Let $i_0=\min\big\{i:\varphi_i(A)\ne A\big\}$, and thus $A'= \varphi_{i_0}(A)$. We first assume that $A'$ is obtained from $A$ with an $i_0$-split, then $A'<A$ and thus by lemma~\ref{le:lowerideal}, $A' \in {\mathcal C}_{\alpha,x}^{\beta,y}$. Moreover, applying an $i_0$-split to $A$ guarantees that an $(i_0-1)$-merge can not be applied to $A'$. The minimality of $i_0$ guarantees that $\varphi_i(A')=A'$ for all $i<i_0$ and $\varphi(A')=\varphi_{i_0}(A')=A$ is obtained from $A'$ by an $i_0$-merge as desired.

Now assume that $A'$ is obtained by an $i_0$-merge. This implies that no part of $A'$ contains (the vertices of) any edge of $G_{\alpha,x}^{\beta,y}$. Hence, $A' \in {\mathcal C}_{\alpha,x}^{\beta,y}$ by lemma~\ref{le:upperideal}. Again, the minimality of $i_0$ guarantees that $\varphi_i(A')=A'$ for all $i<i_0$ and $\varphi(A')=\varphi_{i_0}(A')=A$ is obtained from $A'$ by an $i_0$-split. Finally, notice that in either case, $sgn( \varphi(A))\neq sgn(A)$ whenever $ \varphi(A)\neq A$.
\end{proof}

\subsection{Proof of Theorem~\ref{thm:antiLxP}}\label{ss:proof1} Lemma \ref{le:involution} tells us that every element $A$ in the poset ${\mathcal C}_{\alpha,x}^{\beta,y}$ is either a fixed point, or is paired with a unique element $B\in{\mathcal C}_{\alpha,x}^{\beta,y}$ such that $B$ is a covering of $A$ or $A$ covers it. Thus, equation~\eqref{eq:anticoef} can be rewritten as:
\begin{equation}\label{eq:anticoeffix}
  c_{\alpha,x}^{\beta,y}=\sum_{A\in {\mathcal C}_{\alpha,x}^{\beta,y}  \atop \varphi(A)=A} (-1)^{\ell(A)}.
\end{equation}
This depends only on the structure of the graph $G_{\alpha,x}^{\beta,y}$, which as remarked earlier, is non-nested. In this section we let $G:=G_{\alpha,x}^{\beta,y}$ be a non-nesting graph on the vertices $\{1,2,\ldots, m\}$ and put ${\mathcal C}(G):= {\mathcal C}_{\alpha,x}^{\beta,y}$, $c(G) := c_{\alpha,x}^{\beta,y}$. Our next task is to describe the fixed points of $\varphi\colon {\mathcal C}(G)\to {\mathcal C}(G)$ in order to resolve equation (\ref{eq:anticoeffix}). To this end, we now prove some auxiliary lemmas that show how $c(G)$ is affected by certain properties that the graph $G$ may have.

\begin{definition}\label{def:disconnected}
Let $G$ be as above. We say that $G$ is \emph{disconnected} if there exists a vertex $1\leq r<m$ such that there is no arc $(a,b)\in G$ with $a\in\{1,\dots ,r\}$ and $b\in\{r+1,\dots,m\}$.

\end{definition}

\begin{lemma}\label{le:disconnected}
 If $G$ is disconnected, then $c(G)=0$.
\end{lemma}

\begin{proof}
Let $r$ be as in Definition \ref{def:disconnected}. We construct a different sign reversing involution  $\psi_r\colon {\mathcal C}(G)\to {\mathcal C}(G)$ with no fixed points, this will imply the claim. Let $A=(A_1,\ldots,A_k)\in {\mathcal C}(G)$ and let $r\in A_j$. If $r+1\in A_j$ let 
$$
\psi_r(A_j) :=(A_1,\ldots,A_{j-1},\{\min(A_j),\ldots,r\},\{r+1,\ldots,\max(A_j)\},A_{j+1},\ldots,A_k) .
$$
Thus by Lemma~\ref{le:lowerideal}, $\psi_r(A_j) \in {\mathcal C}(G)$ since $\psi_r(A_j)$ refines $A$. If $r+1\notin A_j$ then $r+1=\min A_{j+1}$. In this case let
$$\psi_r(A)=(A_1,\ldots,A_{j-1},A_j\cup A_{j+1},A_{j+2},\ldots,A_k)$$ Since $G$ is disconnected at $r$ we see that $\psi_r(A_j) \in {\mathcal C}(G) $, as desired. It is not difficult to check that in either case, $\psi_r(\psi_r(A_j))$. This completes the proof.
%
\end{proof}

\begin{lemma}\label{le:shortedge}
 If $(i,i+1)\in G$ for some $1\le i<m$, then 
  $$ c(G)= c\big(G|_{\{1,\ldots, i\}}\big)\cdot c\big(G|_{\{i+1,\ldots, m\}}\big) $$
\end{lemma}
\begin{proof} Let $(i,i+1)\in G$ for some $1\le i<m$, then there is no other edge $(a,b)\in G$ with $a\le i<b$ since $G$ is non-nested. Thus we can think of $G$ as formed by the subgraphs $G' = G|_{\{1,\ldots, i\}}$ and $G''=G|_{\{i+1,\ldots, m\}}$ together with the edge $(i,i+1)$ that connects $G'$ and $G''$. Moreover, notice that in this case the set ${\mathcal C}(G)$ can be thought of as  ${\mathcal C}(G')\times {\mathcal C}(G'')$ since for any $A\in{\mathcal C}(G)$, $i$ and $i+1$ must be separated in $A$.
%
Hence
$$ c(G) = \sum_{A\in {\mathcal C}(G)} (-1)^{\ell(A)}  = \sum_{(A',A'')\in {\mathcal C}(G')\times {\mathcal C}(G'')} (-1)^{\ell(A')+\ell(A'')}= c\big(G|_{\{1,\ldots, i\}}\big)\cdot c\big(G|_{\{i+1,\ldots, m\}}\big) 
$$
as desired. \end{proof}

From Lemma~\ref{le:disconnected} and Lemma~\ref{le:shortedge}, we can assume from now on that $G$ is non-nested, connected and with no short  edges, i.e., edges of the form $(i,i+1)$. In particular, such $G$ must contain an edge $(1,\ell)\in G$ with  $2<\ell\le m$. Moreover, if $\ell=m$ it follows that $G=\{(1,m)\}$ and the only fixed point of $\varphi$ is the set composition $A=(\{1\},\{2,\ldots,m\})$.

Now, assume that $2<\ell<m$. Since $G$ is connected, there must be an edge
$(a,b)\in G$ such that $1<a\le \ell<b\le m$. Consider the set of edges $\big\{(a_1,b_1),\ldots,(a_n,b_n)\big\}\subseteq G$ such that  $$1<a_1<a_2<\ldots<a_n\le\ell<b_1<b_2\ldots< b_n\le m.$$

\begin{lemma}\label{le:nequal1} With $\big\{(a_1,b_1),\ldots,(a_n,b_n)\big\}$ as above, we have that the fixed points of $\varphi$ depend only on $(a_1,b_1)$.
\end{lemma}

\begin{proof} 
Assume that $n>1$ and  that $A=(A_1,\ldots,A_k)\in  {\mathcal C}(G)$ is a fixed point of $\varphi$. We have that $A_1=\{1\}$, otherwise we could perform a 1-split on $A$. Similarly, $A_2=\{2,\ldots,r\}$ and thus $ \ell\le r$, otherwise we could perform a 1-merge on $A$. Also, $\ell\le r<b_1$ as the edge $(a_1,b_1)$ can not be contained in $A_2$. Moreover, $|A_2|>2$ and $\{r+1,\ldots,m\}$ has at least two elements. Thus $A_3=\{r+1,...\}$ is nonempty. If $|A_3|>1$, then we can perform a $r+1$-split which contradicts the choice of $A$. Hence, $A_3=\{r+1\}$. Let $c=r+1$ and  $A_4=\{c+1,\ldots,r'\}$.  If there is no edge $(c,d)\in G$,  then we would be allowed to do a $c$-merge on $A$, contradicting its choice. Thus such an edge $(c,d)$ exists. Since $G$ is non-nested, we have $ 1<a_1\le\ell <c\le b_1<b_n<d\le m$. That is
\begin{equation}\label{eq:firstarcs}
G=
\begin{tikzpicture}[baseline=.2cm]
	\foreach \x in {1,3,5,7,9,11,13,15,17} 
		\node (\x) at (\x/2,0) [inner sep=-1pt] {$\bullet$};
	\node at (1/2,-.2) {$\scriptstyle 1$};
	\node at (2/2,-.2) {$\scriptstyle \cdots$};	
	\node at (3/2,-.2) {$\scriptstyle a_1$};
	\node at (4/2,-.2) {$\scriptstyle \cdots$};	
	\node at (5/2,-.2) {$\scriptstyle a_n$};
	\node at (6/2,-.2) {$\scriptstyle \cdots$};	
	\node at (7/2,-.2) {$\scriptstyle \ell$};
	\node at (8/2,-.2) {$\scriptstyle \cdots$};	
	\node at (9/2,-.2) {$\scriptstyle c$};
	\node at (10/2,-.2) {$\scriptstyle \cdots$};	
	\node at (11/2,-.2) {$\scriptstyle b_1$};
	\node at (12/2,-.2) {$\scriptstyle \cdots$};	
	\node at (13/2,-.2) {$\scriptstyle b_n$};
	\node at (14/2,-.2) {$\scriptstyle \cdots$};	
	\node at (15/2,-.2)  {$\scriptstyle d$};
	\node at (16/2,-.2) {$\scriptstyle \cdots$};	
	\node at (17/2,-.2)  {$\scriptstyle m$};
	\draw [thick] (1) .. controls (3/2,1) and (5/2,1) .. (7); 
	\draw [thick] (3) .. controls (5/2,1) and (9/2,1) .. (11); 
	\draw [densely dotted] (5) .. controls (7/2,.75) and (11/2,.75) .. (13); 
	\draw [densely dotted,thick] (9) .. controls (11/2,1) and (13/2,1) .. (15); 
\end{tikzpicture} 
   \end{equation}
Hence,
 \begin{equation}\label{eq:firstparts}
  A=\big(\{1\},\{2,\ldots,c-1\},\{c\},\{c+1,\ldots,r'\},\ldots\big)
  \end{equation}
where $r'\ge d$. Thus, the fixed point $A$ does not depend on the edges $(a_2,b_2),\ldots,(a_n,b_n)$, and the claim follows.
\end{proof}

The proof of  Lemma~\ref{le:nequal1} gives us a necessary condition on the fixed points of $\varphi$.

\begin{lemma}\label{le:fixshape}   Let $G$ be connected with no small edges.
 If $A\in  {\mathcal C}(G)$ is a fixed point of $\varphi$, then 
    $$A=(\{{\small 1}\},\{2,\ldots,x_2-1\},\{x_2\},\{x_2+1,\ldots,x_4-1\},\ldots,\{x_{2k}\},\{x_{2k}+1,\ldots,m\})$$
    when $\ell(A)$ is even, and
    $$A=(\{1\},\{2,\ldots,x_2-1\},\{x_2\},\{x_2+1,\ldots,x_4-1\},\ldots,\{x_{2k}\},\{x_{2k}+1,\ldots,m-1\},\{m\})$$    
when $\ell(A)$ is odd. In each case $G$ contains, respectively, edges of the form
\begin{align*}
& \big\{(x_0,y_0),(x_1,y_1),\ldots,(x_{2k},y_{2k})\big\}\text{ with }x_0=1 \text{ and }y_{2k} = m,\text{ or},\\
 &\big\{(x_0,y_0),(x_1,y_1),\ldots,(x_{2k},y_{2k}),(x_{2k+1},y_{2k+1})\big\}\text{ with }x_0=1 \text{ and }y_{2k+1} = m
 \end{align*}
  such that 
  \begin{enumerate}
   \item[(i)] for $0\le i\le k-1$ we have $x_{2i}<x_{2i+1}\le y_{2i}<x_{2i+2}\le y_{2i+1}<y_{2i+2}$,
   \item[(ii)] there is no edge $(x,y)\in G$ such that $x_{2i}<x<x_{2i+1}$.
     \end{enumerate}
\end{lemma}

\begin{proof} 

The case where $G$ has only one edge was considered prior to Lemma~\ref{le:nequal1}. In this case, the unique fixed point is $A=(\{1\},\{2,\ldots,m\})$. 

If $G$ has more than one edge,
Lemma~\ref{le:nequal1} tells us that the fixed points of $\varphi$ depend only on edges of the form
$(1,\ell)$, $(a_1,b_1)$ and the possible $(c,d)$ as in equation~\eqref{eq:firstarcs}. If there is no such edge $(c,d)$, then 
$G=\big\{(1,\ell),(a_1,b_1),\ldots,(a_n,b_n)\big\}$, where $1<a_j\le \ell<b_j\le m$ and $b_n=m$.
For $n>1$, we have seen in the proof of Lemma~\ref{le:nequal1} that if there is no arc $(c,d)\in G$ with $ 1<a_1<a_n \le\ell <c\le b_1<b_n<d\le m$, then there is no fixed point of $\varphi$. 
If $n=1$, then $G=\big\{(1,\ell),(a,m)\big\}$ for $1<a\le \ell<m$. Our analysis shows that in this case there is a unique fixed point $A=(\{1\},\{2,\ldots,m-1\},\{m\})$. Here $\ell(A)=3$ is odd, $k=0$ and 
again all the conditions of the lemma are satisfied.

Assume now that $G$ has an edge $(c,d)$ as in equation~\eqref{eq:firstarcs}. Since for $j>1$, the edges $(a_j,b_j)$ do not play a role 
in our analysis of the fixed point of $\varphi$, we can omit them. Let $(a,b)=(a_1,b_1)$ and consider the set of arcs $\big\{(c_1,d_1),\ldots,(c_n,d_n)\big\}\subseteq G$ such that $\ell<c_j\le b<d_j\le m$. We now have
\begin{equation}\label{eq:nextarcs}
G=
\begin{tikzpicture}[baseline=.2cm]
	\foreach \x in {1,3,5,7,9,11,13,15,17} 
		\node (\x) at (\x/2,0) [inner sep=-1pt] {$\bullet$};
	\node at (1/2,-.2) {$\scriptstyle 1$};
	\node at (2/2,-.2) {$\scriptstyle <$};	
	\node at (3/2,-.2) {$\scriptstyle a$};
	\node at (4/2,-.2) {$\scriptstyle \le$};	
	\node at (5/2,-.2) {$\scriptstyle \ell$};
	\node at (6/2,-.2) {$\scriptstyle <$};	
	\node at (7/2,-.2) {$\scriptstyle c_1$};
	\node at (8/2,-.2) {$\scriptstyle <\cdots <$};	
	\node at (9/2,-.2) {$\scriptstyle c_n$};
	\node at (10/2,-.2) {$\scriptstyle \le$};	
	\node at (11/2,-.2) {$\scriptstyle b$};
	\node at (12/2,-.2) {$\scriptstyle <$};	
	\node at (13/2,-.2) {$\scriptstyle d_1$};
	\node at (14/2,-.2) {$\scriptstyle <\cdots <$};	
	\node at (15/2,-.2)  {$\scriptstyle d_n$};
	\node at (16/2,-.2) {$\scriptstyle \le$};	
	\node at (17/2,-.2)  {$\scriptstyle m$};
	\draw [thick] (1) .. controls (2/2,.75) and (4/2,.75) .. (5); 
	\draw [thick] (3) .. controls (5/2,1) and (9/2,1) .. (11); 
	\draw [densely dotted,thick] (7) .. controls (9/2,.75) and (11/2,.75) .. (13); 
	\draw [densely dotted,thick] (9) .. controls (11/2,.75) and (13/2,.75) .. (15); 
\end{tikzpicture} 
   \end{equation}
For each $1\le j<n-1$, a potential fixed point according to Equation~\eqref{eq:firstparts}  would need to be of the form
\begin{equation}\label{eq:startfixpoint}
 A=(\{1\},\{2,\ldots,c_{j}-1\},\{c_{j}\},\{c_{j}+1,\ldots,r\},\{r+1,\ldots\},\ldots)
 \end{equation}
  where $d_j\le r<d_{j+1}$. The second inequality comes from the fact that we are not allowed to have $c_{j+1}$ and $d_{j+1}$ in the same part.
Hence if $A$ is a fixed point it must have the form described in Equation~\eqref{eq:startfixpoint} and we must have
 $$\big\{(x_0,y_0),(x_1,y_1),(x_2,y_2),\ldots\big\}=\big\{(1,\ell),(a,b),(c_j,d_j)\ldots\big\}\subseteq G,$$
where $ 1<a \le\ell <c_j \le b <d_j$ and there is no edge $(x,y)\in G$ such that $1<x<a$. The remaining structure of the fixed point in Equation~\eqref{eq:startfixpoint}
depends only on the structure of the smaller graph $G|_{\{c_j,\ldots,m\}}$. The result then follows by induction on the size of $G$.
\end{proof}

Now that we have a better understanding of the possible structure of the fixed points of $\varphi$, it may appear that there are many possibilities.
It turns out that there could be at most two fixed points of different parity. 

\begin{proof}[Proof of Theorem~\ref{thm:antiLxP}] Let $A$ be a fixed point of $\varphi$. Assume first that $\ell(A)$ is even. Lemma~\ref{le:fixshape} gives that we must have edges $\big\{(x_0,y_0),(x_1,y_1),\ldots,(x_{2k},y_{2k})\big\}\subseteq G$
satisfying the conditions (i), (ii). If $k=0$, then $G=\{(1,m)\}$ and  there is a unique fixed point $A=(\{1\},\{2,\ldots,m\})$. Now assume that $k>0$, in which case $y_{2k}=m$ and the edge $(x_{2k},y_{2k})$ is determined. With $i=k-1$ in condition (i) of Lemma~\ref{le:fixshape} we have
\begin{equation}\label{eq:lastarcs}
G=
\begin{tikzpicture}[baseline=.2cm]
	\foreach \x in {1,3,5,7,9,11,13,15,17} 
		\node (\x) at (\x/2,0) [inner sep=-1pt] {$\bullet$};
	\node at (1/2,-.2) {$\scriptstyle 1$};
	\node at (2/2,-.2) {$\scriptstyle $};	
	\node at (3/2,-.2) {$\scriptstyle \cdots$};
	\node at (4/2,-.2) {$\scriptstyle $};	
	\node at (5/2,-.2) {$\scriptstyle \cdots$};
	\node at (6/2,-.2) {$\scriptstyle $};	
	\node at (7/2,-.2) {$\scriptstyle x_{2k-2}$};
	\node at (8/2,-.2) {$\scriptstyle <$};	
	\node at (9/2,-.2) {$\scriptstyle x_{2k-1}$};
	\node at (10/2,-.2) {$\scriptstyle \le$};	
	\node at (11/2,-.2) {$\scriptstyle y_{2k-2}$};
	\node at (12/2,-.2) {$\scriptstyle <$};	
	\node at (13/2,-.2) {$\scriptstyle x_{2k}$};
	\node at (14/2,-.2) {$\scriptstyle \le$};	
	\node at (15/2,-.2)  {$\scriptstyle y_{2k-1}$};
	\node at (16/2,-.2) {$\scriptstyle <$};	
	\node at (17.5/2,-.2)  {$\scriptstyle y_{2k}=m$};
	\draw [densely dotted,thick] (7) .. controls (8/2,.75) and (10/2,.75) .. (11); 
	\draw [densely dotted,thick] (9) .. controls (11/2,.75) and (13/2,.75) .. (15); 
	\draw [thick] (13) .. controls (14/2,.75) and (16/2,.75) .. (17); 
\end{tikzpicture} ,
   \end{equation}
 and condition (ii) on the edges $(x_{2k-2},y_{2k-2})$, $(x_{2k-1},y_{2k-1})$ must also satisfy the condition (ii) of Lemma~\ref{le:fixshape}. 
 Thus these edges $(x_{2k-1},y_{2k-1})$ and $(x_{2k-2},y_{2k-2})$ are uniquely determined and are such that they bound the vertex $x_{2k}$ on the right and on the left, respectively, i.e. $y_{2k-2}<x_{2k}\le y_{2k-1}$.
In this way, $\big\{(x_{2k-2},y_{2k-2}),(x_{2k-1},y_{2k-1}),(x_{2k},y_{2k})\big\}$ are uniquely determined.
 Now we can repeat the process with $i=k-2, k-3,\ldots,0$ in conditions (ii) and (iii) of  Lemma~\ref{le:fixshape} to successively determine the edges  $\big\{(x_0,y_0),(x_1,y_1),\ldots,(x_{2k},y_{2k})\big\}\subseteq G$, and the partition $A$ is given as in Lemma~\ref{le:fixshape}.  
 
 The case when the fixed point $A$ has odd length is very similar. The  condition (i) of  Lemma~\ref{le:fixshape} gives $y_{2k+1}=m$ and hence determines the edge $(x_{2k+1},y_{2k+1})$ . 
Then condition (iii) of  Lemma~\ref{le:fixshape} with $i=k-1$ determines uniquely (if it exists) $(x_{2k},y_{2k})$ as the rightmost edge of $G$ such that $y_{2k}<y_{2k+1}$. Once $x_{2k}$ is determined
we continue the process as above with $i=k-2, k-3,\ldots,0$ to determine uniquely, if possible, all the other edges. Again, if at any time in the process we fail, then there is no fixed point with $\ell(A)$  odd. If we do not fail, there is a unique fixed point with $\ell(A)$  odd.

In conclusion, there are four possibilities. We could have no fixed point and in this case $c(G)=0$; we could have exactly one fix point of odd length and $c(G)=-1$; we could have exactly one fixed point of even length and $c(G)=1$; or we have exactly two fixed points of different parity each and $c(G)=0$ in that case. In all cases Theorem~\ref{thm:antiLxP} follows. 
\end{proof}

\begin{remark}\label{rem:coefcomp} 

Once a non nesting graph $G$ is given, the value of $c(G)$ is very efficient to compute. Lemma~\ref{le:disconnected} gives us that $c(G)=0$ if $G$ is disconnected. Then we decompose 
$G$ according to Lemma~\ref{le:shortedge} into components $G'$ with no short edges. For each component, we follow the (efficient) procedure in the proof of Theorem~\ref{thm:antiLxP} to determine if there is
and even and/or an odd fixed point. This gives us quickly the value of $c(G')$ for each component $G'$.
\end{remark}

\section{Antipode for commutative linearized Hopf monoid $\bf H$}\label{s:antiH}

In this section we show new formulas for commutative and cocommutative linearized Hopf monoid $\bH$. 
The interest of such a formulas is from the fact that it gives a formula for the Hopf algebra $\Kcb(\bH)$. We also aim to introduce a geometrical interpretation related to our antipode formula in terms of certain faces of a polytope in the spirit of the work of Aguiar-Ardila~\cite{Aguiar-Ardila}.
To achieve this, first we give a formula for the antipode in term of orientations in hypergraphs as in Section~\ref{subsec:acyclic}.
The second part will be done jointly with J. Machacek in a sequel paper and is previewed in Section~\ref{ss:nestohedron}

\subsection{Takeuchi's Formula for $\bH$}\label{ss:UsingTakeushiH}

Let $\bH$ be a Hopf monoid linearized in the basis $\bh$. 
Again, we intend to resolve the cancelation in the Takeuchi formula for $\bH$. For a fixed finite set $I$ let $x\in \bh[I]$. 
From~\eqref{eq:takeuchi} we have
\begin{equation}\label{eq:takeuchiH1}
 S_I(x) =  \sum_{A\models I} (-1)^{\ell(A)} \mu_A \Delta_A(x)
  = \sum_{A\models I \atop \Delta_A(x)\ne 0} (-1)^{\ell(A)} x_A,
\end{equation}
 where for $(A_1,\ldots,A_k)\models I$  
and $\Delta_A(x)\ne 0$ we write $x_A=\mu_A\Delta_A(x) \in \bh[I].$
These are unique elements since $\bH$ is linearized in the basis $\bh$. 
We can thus rewrite equation~\eqref{eq:takeuchiH1} as follows
\begin{equation}\label{eq:takeuchiH2}
 S_I(x) = \sum_{y\in \bh[I]}\left(\sum_{A\models I \atop x_A=y} (-1)^{\ell(A)}\right) y.
\end{equation}
Let 
 $${\mathcal C}_{x}^{y}=\big\{A\models I : x_A=y\big\}$$

So far we have not considered any commutative property of $\bf H$. In general we have no control on the set ${\mathcal C}_{x}^{y}$, but when $\bf H$ is commutative and cocommutative,
our next theorem is a new formula for the antipode of $\bH$. The result and its proof is very similar to analogous results in~\cite{Benedetti-Sagan,Bergeron-Ceballos}.
In order to state it, we need some more notation. Given $x,y\in \bh[I]$ such that ${\mathcal C}_{x}^{y}\ne \emptyset$, choose a fixed minimal element $\Lambda=(\Lambda_1,\Lambda_2,\ldots,\Lambda_m)$  in ${\mathcal C}_{x}^{y}$ under refinement. We will see in Lemma~\ref{lem:minimalcom} that $\Lambda$ is unique up to permutation of its parts and that $\Delta_\Lambda(x)=x_{\Lambda_1}\otimes\cdots\otimes x_{\Lambda_m}$. Cocommutativity and associativity guarantees that for $P=\Lambda_i\subseteq I$ the element $x_P$ defined from $\Delta_{P,I\setminus P}(x)=x_P\otimes x_{I\setminus P}$ will be such that
$x_P=x_{\Lambda_i}$. Recall that a \emph{hypergraph $G$} on a vertex set $V$ is a collection $E$ of subsets of $V$. The elements of $E$ are called \emph{hyperedges} and the hypergraph $G$ is \emph{simple} if every $e\in E$ appears only once. We now define a  simple hypergraph $G_x^y$ on the vertex set $[m]$ 
such that $U\subseteq [m]$ is a hyperedge of $G_x^y$ if and only if
  \begin{equation}\label{eq:Gcom}
      \prod_{i\in U} x_{\Lambda_i} \ne x_{\cup_{i\in U} \Lambda_i}
  \qquad \text{ \bf and }\qquad \forall(P\subset U)\quad \prod_{i\in P} x_{\Lambda_i} = x_{\cup_{i\in P} \Lambda_i} 
  \end{equation}
Up to reordering of the vertices $\{1,2,\ldots,m\}$, commutativity, cocommutativity and Lemma~\ref{lem:minimalcom} will guarantee that $G_x^y$ does not depend on our choice of $\Lambda$.

\begin{theorem}\label{thm:antiH} Under the conditions above 
 \begin{equation}\label{eq:thm2}
 S_I(x) = \sum_{y\in \bh[I]} a(G_x^y) y,
\end{equation}
where $a(G_x^y)$ is a signed sum of acyclic orientations of the hypergraph $G_x^y$ to be defined in Section~\ref{subsec:acyclic}
\end{theorem}

\begin{remark}
If $G_x^y$ is a graph, that is, any hyperedge $U\in G_x^y$ is such that $|U|=2$, then every acyclic orientation will have the same sign, as seen in Example~\ref{ex:graphdone}. Hence the theorem above gives a
cancelation free formula for the antipode as shown in \cite{Humpert-Martin}. In general it will not be cancelation free but it is the best generalization, to our knowledge,
for hypergraphs and to a large class of Hopf monoids and Hopf algebras. \end{remark}

\subsection{Structure of ${\mathcal C}_{x}^{y}$ and its hypergraph $G_x^y$}
Before we prove Theorem \ref{thm:antiH} we need to establish some properties of ${\mathcal C}_{x}^{y}=\big\{A\models I : x_A=y\big\}$. This will allow us to determine the coefficient of $y$ in $S(x)$ given by
\begin{equation} \label{eq:cxycocom}
 c_x^y=\sum_{A\in {\mathcal C}_{x}^{y}} (-1)^{\ell(A)}
 \end{equation}

\begin{lemma}\label{lem:minimalcom}
 If $A$ and $\Lambda$ in $ {\mathcal C}_{x}^{y}$ are two minimal set compositions under refinement, then $A$ is a permutation of the parts of $\Lambda$.
 Conversely, any set composition obtained by a permutation of the parts of $\Lambda$ belongs to ${\mathcal C}_{x}^{y}$ and is minimal.
\end{lemma}
\begin{proof}
Given any $B=(B_1,B_2,\ldots, B_k)\in {\mathcal C}_{x}^{y}$ and any permutation $\sigma\colon[k]\to[k]$, we have that $\sigma(B):=(B_{\sigma(1)},\ldots,B_{\sigma(k)})\in {\mathcal C}_{x}^{y}$.
Indeed, this follow from commutativity and cocommutativity since $x_{\sigma(B)}=x_B=y$. Now if $\Lambda=(\Lambda_1,\ldots,\Lambda_m)\in {\mathcal C}_{x}^{y}$ is a minimal set composition under refinement, then $\sigma(\Lambda)$ is in ${\mathcal C}_{x}^{y}$ for any permutation $\sigma\colon[m]\to[m]$. Furthermore $\sigma(\Lambda)$ must be minimal under refinement for if $B<\sigma(\Lambda)$ such that $B\in {\mathcal C}_{x}^{y}$, then we can find a permutation $\tau$ such that $\tau(B)<\Lambda$ and $\tau(B)\in {\mathcal C}_{x}^{y}$. This would contradict the minimality of $\Lambda$.
This shows the conversely part of the lemma.

Now consider $A=(A_1,\ldots,A_\ell)\in {\mathcal C}_{x}^{y}$ another minimal set composition under refinement. Assume that $A\ne\sigma(\Lambda)$ for any $\sigma$. 
We claim that, there is a rearrangement of the parts of $\Lambda$ and $A$ such that $\emptyset\ne U_1= A_1\cap \Lambda_1\ne \Lambda_1$. 
If not, then for all $i,j$ such that $A_i\cap\Lambda_j\ne\emptyset$ we would have  $A_i\cap\Lambda_j=\Lambda_j$ and this would implies that a permutation of $\Lambda$ is a refinement of $A$, a contradiction. 
We can further rearrange the parts of $\Lambda$ such that $U_i=A_1\cap \Lambda_i\ne\emptyset$ for $1\le i\le r$ and $A_1\cap \Lambda_i=\emptyset$ for $r<i\le m$.
As in Equation~\eqref{eq:Tcutx}, for  $T=A_2\cup\cdots\cup A_\ell$ we have
\begin{equation} \label{eq:Mincutx}
    y=x_A=\mu_{A_1,T}\Delta_{A_1,T}(x_A).
   \end{equation}
Let $V_i=\Lambda_i\setminus U_i=\Lambda_i\cap T$ for $1\le i\le r$.
As in the proof of Lemma~\ref{le:min}, we claim that the set composition 
  $$C=(U_1,V_1,\ldots,U_r,V_r,\Lambda_{r+1}\ldots,\Lambda_m)<\Lambda,$$
(we remove any occurrence of $\emptyset$ parts), and $C$ belong to  ${\mathcal C}_{x}^{y}$. The refinement is strict since $U_1\ne\Lambda_1$ and this contradicts the minimality of $\Lambda$, hence no such $A$ may exists.

To show our last claim, we apply Equation~\eqref{eq:Mincutx} to $x_\Lambda=y=x_A$. Let $\Lambda_{i\cdot\cdot  m}=\Lambda_i\cup\Lambda_{i+1}\cup\cdots\cup \Lambda_m$,
$U_{i\cdot\cdot r}= A_1\cap \Lambda_{i\cdot\cdot  m}$ and $T_{i\cdot\cdot m}= T\cap \Lambda_{i\cdot\cdot  m}$ we have
\begin{equation}\label{eq:ALcutx}
     y=\mu_{A_1,T}\Delta_{A_1,T}(x_\Lambda)=\mu_{A_1,T}\Delta_{A_1,T}\mu_{\Lambda_1,\Lambda_{2\cdot\cdot m}}(\Id_{\Lambda_1}\otimes \mu_{(\Lambda_2,\ldots,\Lambda_m)})\Delta_\Lambda(x).
        \end{equation}
We now apply the Relation~\eqref{e:comp}, associativity and commutativity to obtain
$$ \begin{aligned}
     \mu_{A_1,T}&\Delta_{A_1,T}\mu_{\Lambda_1,\Lambda_{2\cdot\cdot m}} = \\
            &= \mu_{A_1,T}(\mu_{U_1,U_{2\cdot\cdot r}}\otimes \mu_{V_1,T_{2\cdot\cdot m}})(\Id_{U_1}\otimes\tau_{V_1,U_{2\cdot\cdot r}}\otimes \Id_{T_{2\cdot\cdot m}})
                           (\Delta_{U_1,V_1}\otimes \Delta_{U_{2\cdot\cdot r},T_{2\cdot\cdot m}})\\
            &= \mu_{U_1,V_1,U_{2\cdot\cdot r},T_{2\cdot\cdot m}} (\Delta_{U_1,V_1}\otimes \Delta_{U_{2\cdot\cdot r},T_{2\cdot\cdot m}})\\
            &= \mu_{\Lambda_1,\Lambda_{2\cdot\cdot m}} \big( \mu_{U_1,V_1} \Delta_{U_1,V_1}\otimes  \mu_{U_{2\cdot\cdot r},T_{2\cdot\cdot m}}\Delta_{U_{2\cdot\cdot r},T_{2\cdot\cdot m}}\big).\\
    \end{aligned}$$
Putting this back in Equation~\eqref{eq:ALcutx} we get
  $$\begin{aligned}
    y&=\mu_{\Lambda_1,\Lambda_{2\cdot\cdot m}} \big( \mu_{U_1,V_1} \Delta_{U_1,V_1}\otimes  \mu_{U_{2\cdot\cdot r},T_{2\cdot\cdot m}}\Delta_{U_{2\cdot\cdot r},T_{2\cdot\cdot m}}\big)(\Id_{\Lambda_1}\otimes \mu_{(\Lambda_2,\ldots,\Lambda_m)})\Delta_\Lambda(x)\\
      &=\mu_{\Lambda_1,\Lambda_{2\cdot\cdot m}} \big( \mu_{U_1,V_1} \Delta_{U_1,V_1}\otimes  \mu_{U_{2\cdot\cdot r},T_{2\cdot\cdot m}}\Delta_{U_{2\cdot\cdot r},T_{2\cdot\cdot m}}  \mu_{(\Lambda_2,\ldots,\Lambda_m)}\big)\Delta_\Lambda(x).\\
      \end{aligned}
    $$
 If $r=1$, then $U_{2\cdot\cdot r}=\emptyset$ and $\mu_{U_{2\cdot\cdot r},T_{2\cdot\cdot m}}=\Delta_{U_{2\cdot\cdot r},T_{2\cdot\cdot m}}=\Id_{T_{2\cdot\cdot m}}$.
 In this case we get
    $$ \begin{aligned}
    y &=\mu_{\Lambda_1,\Lambda_{2\cdot\cdot m}} \big( \mu_{U_1,V_1} \Delta_{U_1,V_1}\otimes   \mu_{(\Lambda_2,\ldots,\Lambda_m)}\big)\Delta_\Lambda(x)\\
      &=\mu_{\Lambda_1,\Lambda_{2\cdot\cdot m}} ( \mu_{U_1,V_1}\otimes   \mu_{(\Lambda_2,\ldots,\Lambda_m)})(  \Delta_{U_1,V_1}\otimes  \Id_{\Lambda_2}\otimes\cdots\otimes\Id_{\Lambda_m} )\Delta_\Lambda(x)\\
      &=\mu_C\Delta_C(x)=x_C
      \end{aligned}$$
If $r>1$,  then we repeat the process above with  $ \mu_{U_{i\cdot\cdot r},T_{i\cdot\cdot m}}\Delta_{U_{i\cdot\cdot r},T_{i\cdot\cdot m}}  \mu_{(\Lambda_i,\ldots,\Lambda_m)}$ for $2\le i\le r$ and  we obtain
$y=x_C$ in all cases. This shows that $C\in {\mathcal C}_{x}^{y}$ contradicting the minimality of $\Lambda$.
\end{proof}

We now consider the analogue to Lemma~\ref{le:lowerideal} and Lemma~\ref{le:upperideal} for $\bf H$.

\begin{lemma}\label{le:Hlowerideal}
If ${\mathcal C}_{x}^{y}\ne\emptyset$, then for any $A\in{\mathcal C}_{x}^{y}$ and $\Lambda\in{\mathcal C}_{x}^{y}$ minimal, we have that 
$[\Lambda,A]\subseteq {\mathcal C}_{x}^{y}$.
\end{lemma}

If $[\Lambda,A]=\emptyset$, then the Lemma is trivially true. If $\Lambda\le A$, then the proof of the above lemma is exactly as in Lemma~\ref{le:lowerideal}.
This shows that ${\mathcal C}_{x}^{y}$ is a lower ideal in the order $\bigcup_{\sigma\in S_m} [\sigma\Lambda,(I)]$. The next lemma shows us how ${\mathcal C}_{x}^{y}$
is cut out of $\bigcup_{\sigma\in S_m} [\sigma\Lambda,(I)]$

\begin{lemma}\label{le:Hupperideal}
The minimal elements of $\big(\bigcup_{\sigma\in S_m} [\sigma\Lambda,(I)]\big)\setminus {\mathcal C}_{x}^{y}$ are all the  permutations of set compositions of the form
 $$(\bigcup_{i\in U} \Lambda_i,\Lambda_{v_1},\Lambda_{v_2},\ldots,\Lambda_{v_r})$$
for some $U\in\{1,2,\ldots,m\}$, where $r=m-|U|$ and $\{v_1,\ldots,v_r\}=I\setminus U$. 
\end{lemma}

The proof of this lemma is a direct adaptation of the proof of Lemma~\ref{le:upperideal}. It is clear that the upper ideal $\big(\bigcup_{\sigma\in S_m} [\sigma\Lambda,(I)]\big)\setminus {\mathcal C}_{x}^{y}$ is invariant under permutations. Let $B\in \big(\bigcup_{\sigma\in S_m} [\sigma\Lambda,(I)]\big)\setminus {\mathcal C}_{x}^{y}$ be minimal. If $B$ has more than two parts that are not single parts of $\Lambda$, then let $\sigma\in S_m$ be such that $\sigma\Lambda<B$ and proceed as in the second part of the proof on Lemma~\ref{le:upperideal} to reach a contradiction.

We now define the hypergraph $G_x^y$ associated with ${\mathcal C}_{x}^{y}$. For fixed $x$ and $y$, Lemma~\ref{le:Hupperideal} gives us a set of subsets $U\subseteq I$ defining ${\mathcal C}_{x}^{y}$.
  $$G_x^y=\{U\subseteq I : \text{ $U$ minimal, } \prod_{i\in U} x_{\Lambda_i} \ne x_{\bigcup_{i\in U} \Lambda_i}\} $$
In general, $G_x^y$ is a special hypergraph where all edges have cardinality at least $2$ and if $U\in G_x^y$ then for all $U\subset V\subseteq I$, we have $V\not\in G_x^y$. this second property follows from the minimality of the element of $\big(\bigcup_{\sigma\in S_m} [\sigma\Lambda,(I)]\big)\setminus {\mathcal C}_{x}^{y}$. The hypergraph $G_x^y$ is as defined in Equation~\eqref{eq:Gcom}.

\begin{example}\label{ex:hypergraph}
Let $\bf HG$ be as in Section~\ref{ss:hypergraph}. Consider $I= \{a,b,c,d,e\}$ and pick $x=\big\{\{b,c\},\{a,b,e\},\{a,d,e\},\{b,c,e\}\big\}$ and $y=\big\{
\{b,c\}\big\}$ in ${\bf hg}[I]$. We can represent $x$ and $y$ as follows:
$$
x=
\begin{tikzpicture}[scale=.7,baseline=.5cm]
	\node (a) at (0,0) {$\scriptstyle a$};
	\node (b) at (1,.1) {$\scriptstyle b$};
	\node (c) at (2,0) {$\scriptstyle c$};
	\node (d) at (-.5,1) {$\scriptstyle d$};
	\node (e) at (.6,1.2) {$\scriptstyle e$};
	\draw [fill=gray] (.1,.1) .. controls (.5,.3) .. (.9,.2) .. controls (.65,.6) .. (.6,1) .. controls (.5,.5) .. (.1,.1) ; 
	\draw [fill=gray] (0,.15) .. controls (-.05,.6) .. (-.35,1) .. controls (.1,.8) .. (.5,1) .. controls (.1,.6) .. (0,.15) ; 
	\draw [fill=gray] (1.8,.1) .. controls (1.5,.2) .. (1.1,.2) .. controls (1.1,.5) .. (.7,1) .. controls (1.4,.3) .. (1.8,.1) ; 
	\draw (b) .. controls (1.5,-.2) ..   (c); 
\end{tikzpicture} 
\qquad \qquad
y=
\begin{tikzpicture}[scale=.7,baseline=.5cm]
	\node (a) at (0,0) {$\scriptstyle a$};
	\node (b) at (1,.1) {$\scriptstyle b$};
	\node (c) at (2,0) {$\scriptstyle c$};
	\node (d) at (-.5,1) {$\scriptstyle d$};
	\node (e) at (.6,1.2) {$\scriptstyle e$};
	\draw (b) .. controls (1.5,-.2) ..   (c); 
\end{tikzpicture} 
$$
Up to permutation, the minimum refinement of ${\mathcal C}_x^y$
is $\Lambda=(a,bc,d,e)$. Since $\Lambda$ has $4$ parts, the hypergraph $G_x^y$ is build on the set $\{1,2,3,4\}$.
We have that $x_{bc}x_{e}\ne x_{bce}$ and $x_{a}x_dx_e\ne x_{ade}$. Those are the only minimal coarsening of parts of $\Lambda$ that yield
such inequalities. Hence $G_x^y=\big\{\{1,3,4\},\{2,4\}\big\}$. We represent this as follows:
$$G_x^y=
\begin{tikzpicture}[scale=.7,baseline=.5cm]
	\node (1) at (0,.1) {$\scriptstyle 1$};
	\node (2) at (2.5,.5) {$\scriptstyle 2$};
	\node (3) at (0,1) {$\scriptstyle 3$};
	\node (4) at (1,.5) {$\scriptstyle 4$};
	\draw [fill=gray] (.1,.1) .. controls (.25,.5) .. (.15,.9) .. controls (.4,.5) .. (.85,.5) .. controls (.5,.4) .. (.1,.1) ; 
	\draw (4).. controls (1.8,.6) ..(2);
\end{tikzpicture} 
$$
We now identify the set compositions in  $\bigcup_{\sigma\in S_m} [\sigma\Lambda,(I)]$ with the set compositions in $\bigcup_{\sigma\in S_m}[(\sigma(1),\ldots,\sigma(m)),(12\cdots m)]$. There are $4!$ minimal elements with four parts. There are 30 compositions with 3 parts, namely all the permutation of $(12,3,4),(13,2,4),(14,2,3),(23,1,4),(34,1,2)$.
We have removed here all the permutations of $\red{\underline{(24,1,3)}}$. With 2 parts we have all the  permutations of $(123,4), (12,34),(14,23)$ for a total of 6. We have removed the permutations of 
$\red{\underline{(134,2)}}$ and all the coarsenings of permutations of $\red{\underline{(24,1,3)}}$.  Here
  $$c_x^y=24-30+6=0$$
\end{example}

The identification between $\bigcup_{\sigma\in S_m} [\sigma\Lambda,(I)]$ with $\bigcup_{\sigma\in S_m}[(\sigma(1),\ldots,\sigma(m)),(12\cdots m)]$ shows that computing $c_x^y$ is equivalent 
to computing the coefficient of $\epsilon$, the hypergraph on $[m]$ with no edges, in the antipode of $G_x^y$ in the Hopf monoid of hypergraphs. This implies the following theorem:
\begin{theorem}\label{thm:hyper}
  Given a commutative and cocommutative linearized Hopf monoid {\bf H}, let $x,y\in {\bf h}[I]$. We have\footnote{The reader should be aware of the abuse of notation here: on one hand $c_x^y$ is an antipode coefficient in the Hopf monoid {\bf H}, on the other hand $c_{x/y}^\epsilon $ is an antipode coefficient in the Hopf monoid {\bf HG}.}
    $$ c_x^y = c_{x/y}^\epsilon $$
     where $\epsilon$ is the hypergraph on $[m]$ with no edges and $x/y=G_x^y$ is the hypergraph given in~\eqref{eq:Gcom}.
\end{theorem}

\subsection{A different formula for $c_x^y$}\label{ss:antiHwithLxH} This Section is not needed for the proof of Theorem~\ref{thm:antiH}. We only present it as an application of Section~\ref{s:antiLxH}.
Using Remark~\ref{rem:coefcomp} we now give a more efficient  formula to compute
$c_x^y$. For this we  decompose the order ${\mathcal C}_{x}^{y}$ into disjoint suborders, one for each permutation of $S_m$.
As before we identify ${\mathcal C}_{x}^{y}$ with a lower ideal of $\bigcup_{\sigma\in S_m}[\sigma,(12\cdots m)]$.
Given $A=(A_1,A_2,\ldots,A_\ell)\in {\mathcal C}_{x}^{y}$, we obtain a unique refinement $\sigma(A)<A$ by ordering increasingly each  of the $A_i$ and then splitting them into singletons. For example if $A=(\{2,5,7\},\{1\},\{3,4,9\},\{6,8\})$, then $\sigma(A)=(2,5,7,1,3,4,9,6,8)$.
Let
  $${\mathcal C}_{x,\tau}^{y} = \big\{ A\in {\mathcal C}_{x}^{y} : \sigma(A)=\tau \big\}.$$
With these definitions in mind we state the following proposition.
\begin{proposition}\label{lem:Cxydecomp} For any $x,y\in {\bf h}[I]$ such that ${\mathcal C}_{x}^{y}\ne\emptyset$, let $\Lambda=(\Lambda_1,\ldots,\Lambda_m)\in {\mathcal C}_{x}^{y}$ be a fixed minimal element. We have that
 $$ c_{x}^{y} = \sum_{\tau\in S_m} c(G_{12\cdots m,x/y}^{\tau,\epsilon}) $$
 where $\epsilon$ is the hypergraph on $\{1,2,\ldots,m\}$ with no edges and $x/y=G_x^y$ is the hypergraph given in Equation~\eqref{eq:Gcom}.
\end{proposition}

\begin{proof} From the definition above, it is clear that ${\mathcal C}_{x}^{y} = \biguplus_{\tau\in S_m} {\mathcal C}_{x,\tau}^{y}$. For a fixed $\tau$, we have that $A\in {\mathcal C}_{x,\tau}^{y}$ if and only if
 $$\tau=(12\cdots m)|_{A} \qquad \text{  and } \qquad G_x^y\big|_A=\epsilon.$$
This gives
  $$\sum_{A\in {\mathcal C}_{x,\tau}^{y}} (-1)^{\ell(A)} = c(G_{12\cdots m,x/y}^{\tau,\epsilon}) $$
where $c(G_{12\cdots m,x/y}^{\tau,\epsilon})$ is the coefficient of $(\tau,\epsilon)$ in the expansion of $S(12\cdots m,x/y)$ for the Hopf monoid ${\bf L}\times {\bf HG}$ with $\bf HG$ as defined in Scetion~\ref{ss:hypergraph}.
\end{proof}

Theorem~\ref{thm:antiLxP} gives us that $c(G_{12\cdots m,x/y}^{\tau,\epsilon})$ is $0$ or $\pm 1$. Proposition~\ref{lem:Cxydecomp} gives us an interesting
new way to compute antipode, as a sum over permutations instead of a sum of set compositions. 
\begin{example}\label{ex:twotriangles} We compute the coefficient of the hypergraph $\epsilon$ in the antipode $S(x)$ of the hypergraph $x=\big\{\{1,2,4\},\{2,3,4\}\big\}\in\bf{HG}[4]$. Let $\tau=1243$ and recall the construction of the graph $G=G_{1234,x/\epsilon}^{\tau,\epsilon}$  as in Example~\ref{ex:LxHgraph}. This is a graph on the ordered vertex set $1243$ such that there is  an arc $(i,i+1)$ for each descent $\tau(i)>\tau(i+1)$. Also, we draw  an arc $(i,j)$ for each hyperedge $U\in G$ where $i=\min_\tau(U)$ and $j=\max_\tau(U)$ are the minimum and maximum values of $U$ according to the order $\tau$. Then we erase all drawn arcs that contain a nested arc.
With $\tau$ as above, we have the arc $(4,3)$ from the descent of $\tau$ and the arcs $(1,4)$ and $(2,3)$ for the hyperedges  $\{1,2,4\}$ and $\{2,3,4\}$ respectively.  Then we erase the arc $(2,3)$ since it contains the nested arc $(4,3)$. 
The resulting graph is
$$G=G_{1234,x/\epsilon}^{1243,\epsilon}=
\begin{tikzpicture}[baseline=.2cm]
	\foreach \x in {1,2,3,4} 
		\node (\x) at (\x/2,0) [inner sep=-1pt] {$\bullet$};
	\node at (1/2,-.2) {$\scriptstyle 1$};
	\node at (2/2,-.2) {$\scriptstyle 2$};
	\node at (3/2,-.2) {$\scriptstyle 4$};
	\node at (4/2,-.2) {$\scriptstyle 3$};
	\draw (1) .. controls (1.5/2,.75) and (2.5/2,.75) .. (3); 
	\draw [densely dotted]  (2) .. controls (2.5/2,.75) and (3.5/2,.75) .. (4); 
	\draw (3) .. controls (3.25/2,.5) and (3.75/2,.5) .. (4); 
\end{tikzpicture} 
$$
where the dotted arcs correspond to the removed edges. Then we get $$c(G)=c(G|_{124})\cdot c(G|_3)=(1)\cdot(-1)$$
where the first equality comes from Lemma~\ref{le:shortedge} and the second equality follows by Lemma ~\ref{le:fixshape} since the only fixed point adding up to $c(G|_{124})$ is the composition $(1,24)$, which contributes to 1; similarly, the only fixed point adding up to  $c(G|_3)$ is the composition $(3)$ which contributes to $(-1)$.
For most $\tau$ in this example, we get a disconnected graph and Lemma~\ref{le:disconnected} gives us $c(G_{1234,x/\epsilon}^{\tau,\epsilon})=0$ in those cases. For example,
$$G_{1234,x/\epsilon}^{1342,\epsilon}=
\begin{tikzpicture}[baseline=.2cm]
	\foreach \x in {1,2,3,4} 
		\node (\x) at (\x/2,0) [inner sep=-1pt] {$\bullet$};
	\node at (1/2,-.2) {$\scriptstyle 1$};
	\node at (2/2,-.2) {$\scriptstyle 3$};
	\node at (3/2,-.2) {$\scriptstyle 4$};
	\node at (4/2,-.2) {$\scriptstyle 2$};
	\draw [densely dotted](1) .. controls (1.5/2,1) and (3.5/2,1) .. (4); 
	\draw [densely dotted]  (2) .. controls (2.5/2,.75) and (3.5/2,.75) .. (4); 
	\draw (3) .. controls (3.25/2,.5) and (3.75/2,.5) .. (4); 
\end{tikzpicture} 
$$
Removing the dotted arcs produce a disconnected graph, hence the result is zero. The only 
$\tau$ that will contribute non-trivially are $1243, 2341, 3124,4321$ with sign $-1,-1,-1,1$ respectively.
Hence the coefficient of $\epsilon$ in $S(x)$ is $-1-1-1+1=-2$.
\end{example}

\subsection{$c_x^y$ as a signed sum of acyclic orientations of simple hypergraphs.}\label{subsec:acyclic} We now turn to Theorem~\ref{thm:antiH} to get an antipode formula for $c_x^y$ as a signed sum of acyclic orientations of the hypergraph $G_x^y$. If $G_x^y$ is a graph, then we will recover the formula of Humpert-Martin~\cite{Humpert-Martin}. If $G_x^y$ is a more general hypergraph, then the antipode formula may still have cancelation. In fact much more cancelation than in Section~\ref{ss:antiHwithLxH}, but our aim is to
understand this formula geometrically in sequel work (see Section~\ref{ss:nestohedron}). 
Recall that $G_x^y$ is a hypergraph on the vertex set $[m]$ as defined in Equation~\eqref{eq:Gcom}. The ordering of the vertex set depends on a fixed choice of minimal element in ${\mathcal C}_x^y$.

\begin{definition}[Orientation] Given a hypergraph $G$ an {\sl orientation} $(\mathfrak{a},\mathfrak{b})$ of a hyperedge $U\in G$ is a choice of two nonempty subsets $\mathfrak{a},\mathfrak{b}$ of $U$ such that $U=\mathfrak{a}\cup\mathfrak{b}$ and $\mathfrak{a}\cap\mathfrak{b}=\emptyset$. We can think of the orientation of a hyperegde $U$ as current or flow on $U$ from a single vertex of $\mathfrak{a}$ to the vertices in $\mathfrak{b}$ in which case we say that $\mathfrak{a}$ is the {\sl head} of the orientation $\mathfrak{a}\rightarrow\mathfrak{b}$ of $U$. 
 If $|U|=n$, then there are a total of $2^{n}-2$ possible orientations. An  {\sl orientation of $G$} is an orientation of all its hyperedges. Given an orientation ${\mathcal O}$  on $G$, we say that $(\mathfrak{a},\mathfrak{b})\in{\mathcal O}$ if it is the orientation of a hyperedge $U$ in $G$.
\end{definition}

\begin{definition}[Acyclic orientation] Let $G$ be a hypergraph on the vertex set $V$.
Given an orientation $\mathcal O$ of $G$, we construct an oriented graph $G/{\mathcal O}$ as follow. 
We let $V/{\mathcal O}$ be the finest equivalence class of elements of $V$ defined by the heads of $\mathcal O$.
That is the equivalence defined by the transitive closure of the relation $a\sim  a'$ if $a,a'\in \mathfrak{a}$ for some head $\mathfrak{a}$ of $\mathcal O$.
The oriented edge $([a],[b])$ belongs to $G/{\mathcal O}$ for equivalence classes $[a],[b]\in V/{\mathcal O}$  if and only if there is an oriented hyperedge $(\mathfrak{a},\mathfrak{b})$ of $\mathcal O$ such that $a\in \mathfrak{a}$ and $b\in \mathfrak{b}$. An orientation $\mathcal O$ of $G$ is {\sl acyclic} if the oriented graph $G/{\mathcal O}$ has no cycles.
\end{definition}

\begin{example}\label{ex:124_234}
 Let $G=\big\{\{1,2,4\},\{2,3,4\}\big\}$ be a hypergraph on the vertex set $V=\{1,2,3,4\}$. There are $(2^3-2)(2^3-2)=36$ possible orientations of $G$. The orientations ${\mathcal O}=\big\{ (\{4\},\{1,2\}),(\{2,4\},\{3\})\big\}$ is not acyclic. To see this, the set $V/{\mathcal O}=\{\{1\},\{2,4\},\{3\}\}$ and the oriented graph $G/{\mathcal O}$ contain the 1-cycle $([4],[2])$.
Similarly the orientation ${\mathcal O}'=\big\{ (\{4\},\{1,2\}),(\{2,3\},\{4\})\big\}$ is not acyclic. The equivalence set $V/{\mathcal O'}=\{\{1\},\{2,3\},\{4\}\}$ and the oriented graph $G/{\mathcal O'}$ contains the 2-cycle $([4],[2]);([2],[4])$.
The list of all 20 possible acyclic orientation is
$$  \begin{array}{c}
 \scriptstyle
\{ (\{4\},\{1,2\}),(\{4\},\{2,3\})\};\quad \{ (\{4\},\{1,2\}),(\{3\},\{2,4\})\};\quad \{ (\{4\},\{1,2\}),(\{3,4\},\{2\})\};\quad \{ (\{2\},\{1,4\}),(\{3\},\{2,4\})\};\\
 \scriptstyle
 \{ (\{2\},\{1,4\}),(\{2\},\{3,4\})\};\quad \{ (\{2\},\{1,4\}),(\{2,3\},\{4\})\};\quad \{ (\{1\},\{2,4\}),(\{4\},\{2,3\})\};\quad\{ (\{1\},\{2,4\}),(\{3\},\{2,4\})\};\\
\scriptstyle
\{ (\{1\},\{2,4\}),(\{2\},\{3,4\})\};\quad\{ (\{1\},\{2,4\}),(\{2,3\},\{4\})\};\quad \{ (\{1\},\{2,4\}),(\{2,4\},\{3\})\};\quad\{ (\{1\},\{2,4\}),(\{3,4\},\{2\})\};\\
\scriptstyle
\{ (\{1,2\},\{4\}),(\{3\},\{2,4\})\};\quad\{ (\{1,2\},\{4\}),(\{2\},\{3,4\})\};\quad \{ (\{1,2\},\{4\}),(\{2,3\},\{4\})\};\quad\{ (\{1,4\},\{2\}),(\{4\},\{2,3\})\};\\
\scriptstyle
\{ (\{1,4\},\{2\}),(\{3\},\{2,4\})\};\quad\{ (\{1,4\},\{2\}),(\{3,4\},\{2\})\};\quad \{ (\{2,4\},\{1\}),(\{3\},\{2,4\})\};\quad\{ (\{2,4\},\{1\}),(\{2,4\},\{3\})\}.
\end{array}$$
\end{example}

Our next lemma will show that for every set composition $A\in {\mathcal C}_{x}^{y}$ there is a unique acyclic orientation of $G_x^y$. Conversely for any acyclic orientation there is a least one $A=(A_1,A_2,\ldots, A_\ell)\in {\mathcal C}_{x}^{y}$ that gives that orientation. If we denote by ${\mathfrak O}_x^y$ the set of acyclic orientations of $G_x^y$, then we construct bellow a surjective
map $\Omega\colon  {\mathcal C}_{x}^{y} \to {\mathfrak O}_x^y$.
For any  $1\le i\le\ell$, let $A_{i,\ell}=A_i\cup A_{i+1}\cup\cdots\cup A_\ell$ and let $G/\mathcal{O}_{i,\ell}$ be the restriction of  $G/\mathcal{O}$ to the set $A_{i,\ell}$.

\begin{lemma}\label{lem:Omega}
Let $x,y\in {\bf h}[I]$ and consider the hypergraph $G_x^y$ on $V=[m]$. There is a surjective map $\Omega\colon  {\mathcal C}_{x}^{y} \to {\mathfrak O}_x^y$. More precisely,
\begin{enumerate}
 \item[(a)] For any $A=(A_1,A_2,\ldots, A_\ell)\in {\mathcal C}_{x}^{y}$ there is a unique  $\Omega(A)\in {\mathfrak O}_x^y$ such that for any $U\in G_x^y$ the orientation of $U$ is given by $(U\cap A_i,U\setminus A_i)$ where
  $i=\min\{j: A_j\cap U\ne \emptyset\}.$ Furthermore $V/{\Omega(A)}$ is a refinement of $\{A_1,A_2,\ldots,A_\ell\}$.
 \item[(b)] 
 For any  $\mathcal O\in {\mathfrak O}_x^y$, there is a unique  $A_{\mathcal O}=(A_1,A_2,\ldots, A_\ell)\in {\mathcal C}_{x}^{y}$ such that $\{A_1,A_2,\ldots, A_\ell\}=V/{\mathcal O}$ and $A_i$ is the unique source of the restriction  $G/\mathcal{O}_{i,\ell}$  such that $\min(A_i)$ is maximal among the sources of $G/\mathcal{O}_{i,\ell}$.
We have that $\Omega(A_{\mathcal O})=\mathcal O$.
\end{enumerate}
\end{lemma}

\begin{proof}
   For part (a), let $A=(A_1,A_2,\ldots,A_\ell)\in {\mathcal C}_{x}^{y}$. From Theorem~\ref{thm:hyper}, we have that $A$ must break every hyperedge of $G_x^y$. In particular, for any part $A_i$ of $A$ and $U\in G_x^y$, we always have $A_i\cap U\ne U$.
   Hence $(U\cap A_i,U\setminus A_i)$ for $i=\min\{j: A_j\cap U\ne \emptyset\}$ defines a proper orientation of each edge of $G_x^y$. 
   That is, we have an orientation $\mathcal O$ of $G_x^y$. Remark that by construction, each head $\mathfrak{a}$ of $\mathcal O$ is completely included within a part $A_i$ for a unique part $1\le i\le \ell$. This implies that $V/{\mathcal O}$ refines $\{A_1,\ldots,A_\ell\}$ and it allows us to define a function $f\colon V/{\mathcal O}\to \{1,2,\ldots,\ell\}$ where $f([a])=i$ if and only if $[a]\subseteq A_i$.
   By construction of $\mathcal O$, we have that for any $([a],[b])\in G_x^y/{\mathcal O}$ the function values $f([a])<f([b])$.
   This implies that  $G_x^y/{\mathcal O}$ has no cycle. Hence $\mathcal O$ is acyclic.
   
   For part (b), let $\mathcal O$ be an acyclic orientation on $G_x^y$. It is clear that the set composition $A_{\mathcal O}$ is well defined (for example see \cite{Benedetti-Sagan} for $G/{\mathcal O}$). We need to show that part (a) applied to $A_{\mathcal O}$ gives back $\mathcal O$. We have  that $\{A_1,\ldots, A_\ell\}=V/{\mathcal O}$. Hence for any $(\mathfrak{a},\mathfrak{b})\in {\mathcal O}$ we must have $\mathfrak{a}\subseteq A_i$ for some unique $1\le i\le \ell$. 
   We claim that
 $$A_j\cap   \mathfrak{b}\ne \emptyset\ \quad\implies \quad j>i$$
 If not, then there would be $j<i$ such that $A_j\cap \mathfrak{b}\ne \emptyset$.  This means there is an edge from $A_i$ to $A_j$ in $G/{\mathcal O}_{j,\ell}$, which contradicts the fact that $A_j$ is a source of $G/{\mathcal O}_{j,\ell}$, hence $j$ must be such that $j>i$.
\end{proof}

\begin{theorem} \label{thm:anticoco}
For any $x,y\in {\bf h}[I]$ such that ${\mathcal C}_{x}^{y}\ne\emptyset$ we have
  $$c_x^y=\sum_{{\mathcal O} \in {\mathfrak O}_x^y} (-1)^{\ell{(A_{\mathcal O}})} $$
\end{theorem}

\begin{proof}
Our proof will be similar to the one appearing in~\cite{Bergeron-Ceballos}. First we use the the surjective map $\Omega$ from Lemma~\ref{lem:Omega}  to decompose the formula~\eqref{eq:cxycocom}
 $$c_x^y=\sum_{B\in {\mathcal C}_{x}^{y}} (-1)^{\ell(B)} = \sum_{{\mathcal O} \in {\mathfrak O}_x^y}\left(\sum_{B\in {\mathcal C}_{x}^{y} \atop \Omega(B)={\mathcal O}} (-1)^{\ell(B)}\right)$$
For any fixed orientation $\mathcal O$, we thus have to show
  $$ \sum_{B\in {\mathcal C}_{x}^{y} \atop \Omega(B)={\mathcal O}} (-1)^{\ell(B)} = (-1)^{\ell(A_{\mathcal O})} $$
  Let ${\mathcal C}_{x,{\mathcal O}}^{y}=\{B \in {\mathcal C}_{x}^{y}: \Omega(B)={\mathcal O}\}$.
As in ~\cite{Benedetti-Sagan,Bergeron-Ceballos}, we construct a signed reversing involution $\varphi\colon {\mathcal C}_{x,{\mathcal O}}^{y}\to {\mathcal C}_{x,{\mathcal O}}^{y}$ such that
  
  (A) $\varphi(A_{\mathcal O})=A_{\mathcal O}$ is the only fixed point,
  
  (B) for $B\ne A_{\mathcal O}$, we have $\ell(\varphi(B))=\ell(B)\pm 1$.
 
 \noindent If $B\ne A_{\mathcal O}$, then we have from Lemma~\ref{lem:Omega} that each part of $A_{\mathcal O}$ is included in a part of $A$. Let $A_{\mathcal O}=(A_1,A_2,\ldots,A_\ell)$ and $B=(B_1,B_2,\ldots,B_k)$, we define
   $$f_B\colon \{A_1,A_2,\ldots,A_\ell\}\to \{ B_1,B_2,\ldots,B_k\}$$
as the  function such that $A_i\subseteq f(A_i)$. Since $B \ne A_{\mathcal O}$, we have that $f_B\ne Id$. 
Find the smallest $i$ such that $f_B^{-1}(B_i)\ne\{A_i\}$. 
Let $G/\mathcal{O}_{i,\ell}$ be the restriction of  $G/\mathcal{O}$ to the set $A_{i,\ell}$.
All the elements in $f^{-1}_B(B_i)$ are sources in the graph $G/\mathcal{O}_{i,\ell}$. 
By Lemma~\ref{lem:Omega} (b), we have that $\min(A_i)$ is the largest among the sources of $G/\mathcal{O}_{i,\ell}$.
Since $f^{-1}_B(B_i)\ne\{A_i\}$, there must be a source  $A_r \in f^{-1}_B(i)$ such that $\min(A_r)<\min(A_i)$. 
Let $X\in f^{-1}_B(B_i)$ be such that $\min(X)<\min(A_r)$ for all $A_r\in f^{-1}_B(B_i)$.
 We then find the smallest $j\ge i$ such that $A_r\in f^{-1}_{B}(B_j)$ is a source of $G/\mathcal{O}_{i,\ell}$ and $\min(A_r)>\min(X)$. 
 Such $j$ exists since $G/\mathcal{O}_{i,\ell}$ contain one source, namely $A_i$, such that $\min(A_i)>\min(X)$.
We let 
  $$U=\Big\{ Z\in f^{-1}_{B}(B_j) : {\exists \text{$Y$ a source of } G/\mathcal{O}_{i,\ell},\text{ a path from $Y$ to $X$}\atop \text{ and } \min(Y)\le \min(X)}\Big\}$$
If $U=\emptyset$, then $j>i$ since $X\not\in U$. 
In this case we remark that our choice of $j$ implies that for all $A_r\in f^{-1}_{B}(B_j)$, the element $A_r$ is connected to a source $Y$ where $\min(Y)\le \min(X)$.
Hence, there is no edge $(Y,X)$ in $G/\mathcal{O}_{i,\ell}$ where $Y\in f^{-1}_{B}(B_{j-1})$ and $X\in f^{-1}_{B}(B_j)$.
 If $U=\emptyset$, then we define
\begin{equation}\label{a_merge}
 \varphi(B)={B}'= (B_1,\ldots,B_{j-2},B_{j-1}\cup B_{j},B_{j+1},\ldots,B_k).
 \end{equation}
It is clear that $ \varphi(B)={B}'\in {\mathcal C}_{x,{\mathcal O}}^{y}$ with $\ell({B'})=\ell({B})-1$. 
Moreover 
  $$f^{-1}_{B'}(B'_r)= \begin{cases}
              f^{-1}_{B}(B_r) &\text{ if $r<j-1$,}\cr 
              f^{-1}_{B}(B_{j-1})\cup f^{-1}_{B}(B_j)&\text{ if $r=j-1$,}\cr 
              f^{-1}_{B}(B_{r+1})&\text{ if $r>j-1$.}\cr
              \end{cases}$$
Repeating the procedure above for ${B}'$ we will obtain $i',X',j',U'$ in such a way that $i'=i$, $X'=X$, $j'=j-1$ and $U'=f^{-1}_{B}(B_{j-1})\ne\emptyset$.
Now we consider the case when $U\ne\emptyset$. 
Reversing what we did, let $U^c=f^{-1}_{B}(B_j)\setminus U$. 
All the $Z\in U$ are connected to a source $Y$ of $G/\mathcal{O}_{i,\ell}$ with value $\min(Y)\le \min(X)$.
There is no edges of $G/\mathcal{O}_{i,\ell}$ between $U$ and $U^c$. Hence
\begin{equation}\label{a_split}
 \varphi({B})={B}'= (B_1,\ldots,B_{j-1},\bigcup_{Z\in U} Z,\bigcup_{Z'\in U^c} Z',B_{j+1},\ldots,B_k).
\end{equation}
 Remark that now  $\ell({B'})=\ell({B})+1$.
 Moreover
  $$f^{-1}_{B'}(B_r)= \begin{cases}
              f^{-1}_{B}(B_r) &\text{ if $r<j-1$,}\cr 
              U&\text{ if $r=j$,}\cr 
              U^c&\text{ if $r=j+1$,}\cr 
              f^{-1}_{B}(B_{r-1})&\text{ if $r>j+1$.}\cr
              \end{cases}$$
and for this ${B}'$ we will obtain $i',X',j',U'$ in such a way that $i'=i$, $X'=X$, $j'=j+1$ and $U'=\emptyset$. The map $ \varphi$ is thus the desired involution.
\end{proof}

\begin{example}
  Let us revisit the example~\ref{ex:twotriangles} in the Hopf monoid of hypergraphs. Let $x=\big\{\{1,2,4\},\{2,3,4\}\big\}$ be a hypergraph on the vertices $\{1,2,3,4\}$.
  A full computation of the antipode gives us 
    $$S(
\begin{tikzpicture}[scale=.7,baseline=.2cm]
	\node (a) at (0,0) {$\scriptstyle 2$};
	\node (b) at (1,.1) {$\scriptstyle 1$};
	\node (d) at (-.5,1) {$\scriptstyle 3$};
	\node (e) at (.6,1.2) {$\scriptstyle 4$};
	\draw [fill=gray] (.1,.1) .. controls (.5,.3) .. (.9,.2) .. controls (.65,.6) .. (.6,1) .. controls (.5,.5) .. (.1,.1) ; 
	\draw [fill=gray] (0,.15) .. controls (-.05,.6) .. (-.35,1) .. controls (.1,.8) .. (.5,1) .. controls (.1,.6) .. (0,.15) ; 
\end{tikzpicture} 
)= - \begin{tikzpicture}[scale=.7,baseline=.2cm]
	\node (a) at (0,0) {$\scriptstyle 2$};
	\node (b) at (1,.1) {$\scriptstyle 1$};
	\node (d) at (-.5,1) {$\scriptstyle 3$};
	\node (e) at (.6,1.2) {$\scriptstyle 4$};
	\draw [fill=gray] (.1,.1) .. controls (.5,.3) .. (.9,.2) .. controls (.65,.6) .. (.6,1) .. controls (.5,.5) .. (.1,.1) ; 
	\draw [fill=gray] (0,.15) .. controls (-.05,.6) .. (-.35,1) .. controls (.1,.8) .. (.5,1) .. controls (.1,.6) .. (0,.15) ; 
\end{tikzpicture} 
+2\  \begin{tikzpicture}[scale=.7,baseline=.2cm]
	\node (a) at (0,0) {$\scriptstyle 2$};
	\node (b) at (1,.1) {$\scriptstyle 1$};
	\node (d) at (-.5,1) {$\scriptstyle 3$};
	\node (e) at (.6,1.2) {$\scriptstyle 4$};
	\draw [fill=gray] (0,.15) .. controls (-.05,.6) .. (-.35,1) .. controls (.1,.8) .. (.5,1) .. controls (.1,.6) .. (0,.15) ; 
\end{tikzpicture} 
+2\ \begin{tikzpicture}[scale=.7,baseline=.2cm]
	\node (a) at (0,0) {$\scriptstyle 2$};
	\node (b) at (1,.1) {$\scriptstyle 1$};
	\node (d) at (-.5,1) {$\scriptstyle 3$};
	\node (e) at (.6,1.2) {$\scriptstyle 4$};
	\draw [fill=gray] (.1,.1) .. controls (.5,.3) .. (.9,.2) .. controls (.65,.6) .. (.6,1) .. controls (.5,.5) .. (.1,.1) ; 
\end{tikzpicture} 
-2\ \begin{tikzpicture}[scale=.7,baseline=.2cm]
	\node (a) at (0,0) {$\scriptstyle 2$};
	\node (b) at (1,.1) {$\scriptstyle 1$};
	\node (d) at (-.5,1) {$\scriptstyle 3$};
	\node (e) at (.6,1.2) {$\scriptstyle 4$};
\end{tikzpicture} 
$$
The coefficient $-2$ in front of the empty hypergraph $\epsilon$ was computed in Example~\ref{ex:twotriangles} using $4!$ permutations.
Here we will do so using Theorem~\ref{thm:antiH} and the 20 acyclic orientations of Example~\ref{ex:124_234}.
Lemma~\ref{lem:Omega}~(b) tells us that each of those orientations is paired with each of the following 20 set compositions (respectively)
$$  \begin{array}{c}
 \scriptstyle
(4,3,2,1);\quad (3,4,2,1);\quad (34,2,1);\quad (3,2,4,1);\\
 \scriptstyle
 (2,4,3,1);\quad(23,4,1);\quad (1,4,3,2);\quad (3,1,4,2);\\
\scriptstyle
(1,2,4,3);\quad (1,23,4);\quad (1,24,3);\quad (1,34,2);\\
\scriptstyle
(3,12,4);\quad (12,4,3);\quad (123,4);\quad (14,3,2);\\
\scriptstyle
(3,14,2);\quad (134,2);\quad (3,24,1);\quad (24,3,1).
\end{array}$$
There are $9$ even length set compositions in this list and $11$ odd length. The coefficient is indeed $9-11=-2$.
For the coefficient of $x$ in $S(x)$, we remark that $x/x$ is a single point with no edges. There is a unique orientation of $x/x$ 
and it is represented by a set composition with a single part. Thus the coefficient is $-1$. For $y=\big\{\{1,2,4\}\big\}$,
$x/y$ is a graph on two vertices, say $u$ and $v$, with a single edge between them and thus it has two acyclic orientations, which correspond to the set compositions $(u,v),\,(v,u)$. Hence the coefficient is $2$. The same argument applies for $y'=\big\{\{2,3,4\}\big\}$.
\end{example}

\section{Some applications with Hopf algebras}\label{s:HopfAlgebras}

In this section, we will consider some examples of antipodes corresponding to some combinatorial Hopf algebras. We recover results from \cite{Humpert-Martin,Aguiar-Ardila,BakerJarvisBergeronThiem, Benedetti-Sagan, BenedettiHallamMachacek}, and derive some new formulas.

\subsection{Antipode in the commutative case $H=\Kcb(\bH)$}\label{ss:anticom} We now consider some commutative and cocommutative Hopf monoid $\bf H$ and look at the antipode of $H=\Kcb(\bH)$.

\begin{example} Consider the Hopf monoid $\bPi$ in Section~\ref{ss:partition} and the basis $\mathbf \pi$. Given a set partition $X\in{\mathbf \pi}[I]$
and any set composition $A\models I$ we have that $X_A=X$ if a permutation of $X$ refines $A$, and $X_A=0$ otherwise. This means that the only term in $S(X)$ is $X$
and its coefficient is $c_X^X$. A minimal $\Lambda$ in ${\mathcal C}_X^X$ is $X$ with some ordering of its parts. The hypergraph $G_X^X$ has $m=|X|$ vertices
and no hyperedges. If we use Theorem~\ref{thm:anticoco}, there is a unique orientation of $G_X^X$ and its sign is $(-1)^m$.

If instead we use Proposition~\ref{lem:Cxydecomp}, we sum over the permutations $\tau$ of $m$ where $G_{12\cdots m,\epsilon}^{\tau,\epsilon}$ has only short edges $(i,i+1)$
for each descent $\tau(i)>\tau(i+1)$ of $\tau$. This graph is disconnected unless $\tau=(m,m-1,\ldots,2,1)$ for which $c(G_{12\cdots m,\epsilon}^{\tau,\epsilon})=(-1)^m$. 

The Hopf algebra $\Kcb(\bPi)$ is the space of symmetric functions $Sym$ and the basis $\Kcb(X)=p_{type(X)}$ is the power sum basis where $type(X)=(|X_1|,|X_2|,\ldots,|X_m|)$ written in decreasing order. This gives the well known antipode formula $S(p_\lambda)=(-1)^{\ell(\lambda)}p_{\lambda}$.
\end{example}

\begin{example} \label{ex:graphdone}
Consider the Hopf monoid $\bG$ from Section~\ref{ss:graph} with basis $\mathbf g$. Given a graph $x\in{\mathbf g}[I]$
and $A\models I$ we have that $x_A=y$ is a subgraph of $x$. In fact, $y$ is a flat of $x$.
A minimal element $\Lambda$ in ${\mathcal C}_x^y$ is given by any ordering of the equivalence relation $I/y$ where $a,b\in I$ are equivalent whenever there is a path in $y$ connecting them. 
The hypergraph $G_x^y$ is the simple graph $x/y$, obtained by contracting $x$ along the edges of $y$. It is a graph on the vertex set $V=I/y$ and edges $\{[a],[b]\}$ whenever $[a]\ne [b]$ and there is an edge $\{a',b'\}$ in $x$ such that $a'\in[a]$ and $b'\in[b]$. Since $G_x^y$ has no hyperedges $U$ such that $|U|>2$, all orientations $\mathcal O$ of $G_x^y$ 
are such 
that $V/{\mathcal O}=V$, since the head of each edge has cardinality 1. Hence in Theorem~\ref{thm:anticoco} we have that $\ell(A_{\mathcal O})=|V|=|I/y|$ for each $\mathcal O$.
No cancelation occurs and we recover the formula of~\cite{Humpert-Martin, Benedetti-Sagan}.
\end{example}

\begin{example} \label{ex:simplicial}
We can extend the previous example to a Hopf monoid $\mathbf{SC}$ of abstract simplicial complexes. A simplicial complex on a set $I$ is a collection $x\in 2^I$ such that
  $$V\in x  \implies U\in x,\;\;\;\; \forall U\subseteq V,\; |U|>1.$$
In this way, simplicial complexes extend the notion of graphs and it is a subfamily of hypergraphs. Now let $\mathbf{sc}[I]$ be the linear span of all simplicial complexes on $I$.
The product and coproduct of $\bf HG$, as defined in~\ref{ss:hypergraph}, restricts well to $\mathbf{SC}$. Hence, $\mathbf{SC}$ is a monoid of abstract simplicial complexes
with basis $\mathbf{sc}$. 

Given a $x\in{\mathbf {sc}}[I]$ and any set composition $A\models I$ we have that $x_A=y$ is a simplicial subcomplex of $x$.
A minimal element $\Lambda$ in ${\mathcal C}_x^y$ is given by any ordering of the equivalence relation $I/y$ where $a,b\in I$ are equivalent whenever there is a path in $y$ connecting them. 
The hypergraph $G_x^y$ is the simple graph given by the 1-skeleton of $x/y$, where $x/y$ is obtained by contracting $x$ along the all hyperedeges of $y$. It is a graph on the vertex set $V=I/y$ and edges $\{[a],[b]\}$ whenever $[a]\ne [b]$ and there is an edge $\{a',b'\}$ in $x$ such that $a'\in[a]$ and $b'\in[b]$. Since $G_x^y$ is a simple graph, all orientations $\mathcal O$ of $G_x^y$ 
are such 
that $V/{\mathcal O}=V$. Hence Theorem~\ref{thm:anticoco} gives $\ell(A_{\mathcal O})=|V|=|I/y|$ for each $\mathcal O$.
No cancelation occurs and we recover the formula of~\cite{BenedettiHallamMachacek}.
\end{example}

\begin{remark} \label{ex:Hyper}
As seen in Examples \ref{ex:graphdone} and \ref{ex:simplicial}. the antipode formula in the Hopf monoid $\bf SC$  is a lifting of the antipode in $\bG$. Thus, it is natural to ask if such a lifting can be done to find an antipode formula in $\bf HG$. This case, however, is more intricate as lots of cancelation occur in Theorem~\ref{thm:anticoco}. Many of these cancelations will be resolved 
in future work (see Section~\ref{ss:nestohedron}).
\end{remark}

\begin{example} \label{ex:Hypertree} (suggested to us by J. Machacek)
Let us consider the family of \emph{hypergraph forests}.
Given a hypergraph $h$ on $I$, we say that $(a_0,a_1,\ldots,a_\ell)$ is a path of $h$ if $\{a_i,a_{i+1}\}\subset U_i\in h$ for each $0\le i<\ell$.
We say that a path $(a_0,a_1,\ldots,a_\ell)$ is \emph{proper} if  all the hyperedges $U_i$ are distinct.
A proper cycle in $h$ is a proper path $(a_0,a_1,\ldots,a_\ell)$ such that $a_0=a_\ell$.
A hypergraph $f$ is a \emph{hyperforest} if it does not contain proper cycles. Let ${\bf hf}[I]$ be the set of hyperforest on $I$.
It is not hard to check that the operations of $\bf HG$ restrict properly in the subset of hyperforests.
Hence we have $\bf HF$ the hopf submonoid of hyperforests of $\bf HG$ with basis $\bf hf$.

Given a $x\in{\mathbf hf}[I]$ and any set composition $A\models I$ we have that $f_A=h$ is a subforest of $f$.
A minimal element $\Lambda$ in ${\mathcal C}_f^h$ is given by any ordering of the equivalence relation $I/h$ where $a,b\in I$ are equivalent whenever there is a path in $h$ connecting them. 
The hypergraph $G_f^h$ is the hyperforest given by $f/h$, the contraction of $f$ along all the hyperedeges of $h$. 
Any two vertices of $G_f^h$ that are connected, will be so via a unique proper path.
Since it is a hyperforest, any orientation of $G_f^h$ is acyclic and will contribute to the computation of the coefficient in Theorem~\ref{thm:anticoco}.
\begin{proposition} Let $k=|G_f^h|$ and $\ell$ be the number of connected component of $G_f^h$:
   $$c_f^h = \begin{cases} 
      (-1)^\ell (-2)^k& \text{if   $\forall U\in G_f^h$ we have $|U|$ is even,}
     \\
     0& \text{otherwise.}
   \end{cases}
     $$
\end{proposition}
\begin{proof} This result would follow easily from our understanding of the antipode in Section~\ref{ss:nestohedron}. But to be self contained,  we give here a different proof based on a signed reversing involution on acyclic orientation of $G_f^h$. As in Theorem~\ref{thm:anticoco} let ${\mathfrak O}_f^h$ denote the set of acyclic orientation of $G_f^h$. As we remarked above, for a hyperforest, this are all the orientations of $G_f^h$. 

Recall from Section~\ref{subsec:acyclic} that $G_f^h$ is thought of as a hypergraph (here a hyperforest) on $V=[m]$.
We now define 
a signed reversing involution $\Phi\colon {\mathfrak O}_f^h\to {\mathfrak O}_f^h$. 
Given an orientation $\mathcal O$ of $G_f^h$, 
if possible, find the largest element $z\in V$ such that for some $({\frak a},{\frak b})\in {\mathcal O}$ we have $z=\max({\frak a}\cup {\frak b})$ and
\begin{equation}\label{eq:thezcond}
    ( z\in {\frak a}, \quad |{\frak a}|>1 )\quad \text{or} \quad ( z\in {\frak b}, \quad |{\frak b}|>1 ).
    \end{equation}

Then choose $({\frak a},{\frak b})\in {\mathcal O}$ such that ${\frak a}\cup{\frak b}$ is lexicographically maximal among the hyperedges that satisfy~\eqref{eq:thezcond}. We then define  $\Phi({\mathcal O})={\mathcal O}'$ where ${\mathcal O}'$ is obtained from ${\mathcal O}$
after replacing $({\frak a},{\frak b})$ by $({\frak a}\setminus\{z\},{\frak b}\cup\{z\})$ if $z\in \frak a$, or $({\frak a}\cup\{z\},{\frak b}\setminus\{z\})$ otherwise. It is clear that $\Phi$ is an involution that toggles the maximal element of the orientation of a hyperegde between the two situations in~\eqref{eq:thezcond}. If no such $z$ exists, then define $\Phi({\mathcal O})={\mathcal O}$. 

We now show that $\Phi$ reverses sign, except in its fixed points. First recall from Lemma~\ref{lem:Omega} that $\ell(A_{\mathcal O})=|V/{\mathcal O}|$.
Assume $\Phi({\mathcal O})={\mathcal O}'\ne {\mathcal O}$ and let $z$ and $({\frak a},{\frak b})$ be as above. 
In the situation where $z\in {\frak a}$, we now have $({\frak a}\setminus\{z\},{\frak b}\cup\{z\})\in{\mathcal O}'$ and the rest of the orientations are the same as in $\mathcal O$.
Since there exists a unique proper path between any two equivalent vertices in the 
 equivalent classes $[{\frak a}]_{\mathcal O}$ containing $z$ in $V/{\mathcal O}$, this class will break in exactly two classes $[{\frak a}\setminus \{z\}]_{\mathcal O'}$ and 
$[\{z\}]_{\mathcal O'}$ in $V/{\mathcal O'}$. All the other classes of $V/{\mathcal O}$ and $V/{\mathcal O'}$ remain the same. Hence $(-1)^{\ell(A_{\mathcal O})}=-(-1)^{\ell(A_{\mathcal O'})}$
and $\Phi$ is signed reversing in this case. The argument in the other case is similar.

The involution $\Phi$ reduces the identity in Theorem~\ref{thm:anticoco} and we obtain
  $$ c_f^h = \sum_{{\mathcal O}\in {\mathfrak O_f^h} \atop \Phi({\mathcal O})={\mathcal O}} (-1)^{\ell(A_{\mathcal O})}.$$
  To finish the proof, we need to describe now the fixed points of $\Phi$. If there is no $z$ satisfying equation~\eqref{eq:thezcond}, then for all $({\frak a},{\frak b})\in {\mathcal O}$ and  $z=\max({\frak a}\cup {\frak b})$ we have 
  \begin{equation}\label{eq:fixed}
      {\frak a}=\{z\} \quad \text{or} \quad {\frak b}=\{z\},
  \end{equation}
  and the orientations $\mathcal O$ that satisfy~\eqref{eq:fixed} are the fixed points of $\Phi$. If $|G_f^h|=0$, that is  $G_f^h$ has $m=\ell$ vertices and no hyperedges, then $c_f^h=(-1)^\ell$ as desired. If $|G_f^h|>0$, then let $U\in G_f^h$ be any fixed hyperedge. For instance, pick $U$ to be lexicographically maximal in $G_f^h$ and let $z=\max(U)$.
  In any orientation of $G_f^h$  fixed by $\Phi$ we can toggle the orientation of $U$ between the two situations  in~\eqref{eq:fixed} and still get a fixed point of $\Phi$.
  That is, we can pair all the fixed point of $\Phi$ as $({\mathcal O},{\mathcal O}')$ where ${\mathcal O}\ne{\mathcal O}'$ and they differ only by the orientation of $U={\frak c}\cup\{z\}$ with $({\frak c},\{z\})\in {\mathcal O}$ and $(\{z\},{\frak c})\in {\mathcal O'}$. Using again the fact that there is at most a unique proper path between any vertices in $G_f^h$,
  the elements of  ${\frak c}$ are in a single equivalence class in $V/{\mathcal O}$ and in distinct equivalent classes in $V/{\mathcal O'}$. Hence $|V/{\mathcal O'}|=|V/{\mathcal O}|+|U|-2$.
 we now have 
    $$ c_f^h = \sum_{{\mathcal O}\in {\mathfrak O_f^h} \atop \Phi({\mathcal O})={\mathcal O}} (-1)^{\ell(A_{\mathcal O})}
    = \sum_{({\mathcal O},{\mathcal O}')} (-1)^{\ell(A_{\mathcal O})} (1+(-1)^{|U|}).
      $$
Let us denote by $G_f^{h\cup U}$ the hyperforest obtained by contracting the hyperedge $U$ in $G_f^h$. There is a clear correspondence between the orientation ${\mathcal O}$ of $G_f^h$
and the orientation $\mathcal O''$ of $G_f^{h\cup U}$ together with an orientation of $U$. This is true only for hyperforest as there is a unique proper path between any two vertices.
We thus have  
    $$ c_f^h     = \sum_{({\mathcal O},{\mathcal O}')} (-1)^{\ell(A_{\mathcal O})} (1+(-1)^{|U|})
     =- (1+(-1)^{|U|}) \sum_{{\mathcal O''}\in {\mathfrak O_f^{h\cup U}} \atop \Phi({\mathcal O''})={\mathcal O''}} (-1)^{\ell(A_{\mathcal O''})}.
      $$
The negative sign in the second equality comes from the fact that in contracting $U$ joint together the class $[z]_{\mathcal O}$ and $[{\frak c}]_{\mathcal O}$.
  We now have that
    $ c_f^h = -(1+(-1)^{|U|}) c_f^{h\cup U} $. The proposition follows by induction. If $|U|$ is odd, then we get zero. If $|U|$ is even,
  then we get a contribution of $-2$ for that edge and the induction end with an empty hypergraph with the same number of connected component as $G_f^h$.
 \end{proof}
\end{example}

\subsection{Antipode of\, $ \Kcb(\bL\times\bH) \cong \Kc(\bH)$ for linearized $\bH$}\label{s:KLxH} 
 As we noticed in Section~\ref{ss:K_Kbar} we have that $\Kcb(\bL\times\bH) \cong \Kc(\bH)$ for
any Hopf monoid $\bH$. Given $(\alpha,x)\in (\bL\times\bH)[n]_{S_n}$, the isomorphism is explicitly given by the map $(\alpha,x)\mapsto \bH[\alpha^{-1}](x)$ where $\alpha^{-1}\colon[n]\to[n]$
is the unique bijection such that $\alpha^{-1}(\alpha)=12\cdots n$ and $\bH[\alpha^{-1}](x)\in \bH[n]$ is the image of $x$ under the bijection $\bH[\alpha^{-1}]\colon \bH[n]\to\bH[n]$ obtain via the functor $\bH$.
Since $\Kcb$ preserves antipode, in the case where $\bH$ is linearized, Theorem~\ref{thm:antiLxP} gives us the following formula. For $x\in \bH[n]$
 \begin{equation}\label{eq:SforKH}
   S(x)= \sum_{(\beta,y)\in ({\bf l}\times\bh)[n]} c_{12\cdots n,x}^{\beta,y} \bH[\beta^{-1}](y) = \sum_{z\in\bh[n]}\Big( \sum_{\beta\in {\bf l}[n]}  c_{12\cdots n,x}^{\beta, \bH[\beta](z)} \Big) z\,.
 \end{equation}
Here we have identified the linear order $\beta\in{\bf l}[n]$ and the bijection $\beta=(\beta^{-1})^{-1}\colon[n]\to[n]$ in the notation $\bH[\beta](z)$. From Theorem~\ref{thm:antiLxP}
we have that the $c_{12\cdots n,x}^{\beta, \bH[\beta](z)}$ are $\pm 1$, but further cancelation may occur in Equation~\eqref{eq:SforKH}. It is not the best formula in most cases but it is definitely a big improvement on Takeuchi's formula.

\begin{example} \label{ex:super}
Consider the Hopf monoid $\bPi$ from Section~\ref{ss:partition}.  As seen in~\cite{BakerJarvisBergeronThiem}, The Hopf algebra $\Kc(\bPi)$ is the space of symmetric functions
in non-commutative variables. Our formula~\eqref{eq:SforKH} is cancelation free in this case as all the non-zero terms have the same sign (see  Corollary 4.9 of~\cite{BakerJarvisBergeronThiem} for more details).
\end{example}

\begin{example} \label{ex:PR}
Consider now the Hopf monoid $\bL$ in Section~\ref{ss:order}.  The Hopf algebra $PR=\Kc(\bL)$ was introduced by Patras-Reutenauer~\cite{PR:2004} 
and is also studied under the name $R\Pi$ in~\cite{AM:2010}.
The antipode formula~\eqref{eq:SforKH} for $PR$ gives us that for $\alpha\in{\bf l}[n]$:
 \begin{equation}\label{eq:SforKL}
   S(\alpha)=\sum_{\gamma\in{\bf l}[n]}\Big( \sum_{\beta\in {\bf l}[n]}  c_{\epsilon,\alpha}^{\beta, \beta\circ\gamma} \Big) \gamma\,
 \end{equation}
 where $\epsilon=12\ldots n$ is the identity permutation.
In this example, $\beta=(\beta_1,\ldots,\beta_n)\in {\bf l}[n]$ can be encoding three different objects depending on the context.
It is first the total order $\beta_1<\beta_2<\ldots<\beta_n$ on the points $1,2,\ldots,n$.In~\eqref{eq:SforKL}, when we write $\beta\circ\gamma$, we consider
$\beta$ as the permutation defined by $\beta(i)=\beta_i$. Hence $\beta\circ\gamma=(\beta(\gamma_1),\beta(\gamma_2),\ldots,\beta(\gamma_n))$.
Bellow we will consider $\beta$ and $\beta\circ\gamma$ as encoding the set composition $(\{\beta_1\},\ldots,\{\beta_n\})$ and $(\{\beta(\gamma_1)\},\ldots,\{\beta(\gamma_n)\})$.
These conventions should be clear from the context.
We now need a complete description of
$c_{\epsilon,\alpha}^{\beta, \beta\circ\gamma}$ in order to understand (\ref{eq:SforKL}). The set ${\mathcal C}_{\epsilon,\alpha}^{\beta, \beta\circ\gamma}\ne \emptyset$ if and only if the minimal element $\Lambda$ of ${\mathcal C}_{\epsilon,\alpha}^{\beta, \beta\circ\gamma}$ exists and it is the finest set composition such that
\begin{enumerate}
 \item $\beta\le\Lambda$ and $\beta$ is  increasing  with respect to $\epsilon$ within each part of $\Lambda$,
 \item $\beta\circ\gamma\le\Lambda$ and $\beta\circ\gamma$  is  increasing  with respect to $\alpha$ within each part of $\Lambda$.
 \end{enumerate}
These conditions follow from the proof of Lemma~\ref{le:min} in the case of $\bL\times\bL$. 
Let $A=\beta\lor(\beta\circ\gamma)$ be the finest set composition such that $\beta\le A$ and $\beta\circ\gamma\le A$.
We must have that $A\le \Lambda$. Now, if $\beta$ is not increasing with respect to $\epsilon$ within each part of $A$, then it would not be true for $\Lambda$ either. Similarly
if  $\beta\circ\gamma$  is not increasing  with respect to $\alpha$ within each part of $A$, then it would not be true for $\Lambda$. Hence we have that ${\mathcal C}_{\epsilon,\alpha}^{\beta, \beta\circ\gamma}\ne \emptyset$ if and only if $\Lambda=\beta\lor(\beta\circ\gamma)$ is such that $\beta$ is increasing with respect to $\epsilon$ within each part of $\Lambda$ and $\beta\circ\gamma$  is  increasing  with respect to $\alpha$ within each part of $\Lambda$.

For instance, if  $\alpha=(5,2,1,3,4)$, $\beta=(2,1,3,5,4)$ and $\beta\circ\gamma=(2,5,1,3,4)$, then we see that $\Lambda=\beta\lor(\beta\circ\gamma)=(2,135,4)$. The elements $1,3,5$ of $\beta$ are increasing with respect
to $\epsilon=(1,2,3,4,5)$ within the part $135$ of  $\Lambda$. In $\beta\circ\gamma$ these elements are in the order $5,1,3$ which is increasing with respect to $\alpha$. Hence in this little example
 ${\mathcal C}_{\epsilon,\alpha}^{\beta, \beta\circ\gamma}\ne \emptyset$ and $\Lambda=(2,135,4)$ is the minimum. If we take a different $\gamma'$ so that $\beta\circ\gamma'=(2,5,3,1,4)$,
 now $\Lambda=\beta\lor(\beta\circ\gamma')=(2,135,4)$ but the element $5,3,1$ are not in increasing order with respect to $\alpha$, hence ${\mathcal C}_{\epsilon,\alpha}^{\beta, \beta\circ\gamma'}= \emptyset$.

Now we remark that the number of parts of $\Lambda=\beta\lor(\beta\circ\gamma)$ depend only on $\gamma$ and not $\beta$. This follows from the simple fact that
 $$ \beta\lor(\beta\circ\gamma) = \beta\circ\big( \epsilon\lor\gamma\big).$$
 Hence $\ell\big(\beta\lor(\beta\circ\gamma)\big)=\ell\big( \epsilon\lor\gamma\big)=m$ and if ${\mathcal C}_{\epsilon,\alpha}^{\beta, \beta\circ\gamma}\ne \emptyset$, then the graph $G_{\epsilon,\alpha}^{\beta, \beta\circ\gamma}$ is a graph on the vertex set $[m]$ 
 with an edge $(i,i+1)$ if and only if $\max_\epsilon(\Lambda_i)>_\epsilon\min_\epsilon(\Lambda_{i+1})$ or $\max_\alpha(\Lambda_i)>_\alpha \min_\alpha(\Lambda_{i+1})$, using the order $\epsilon$ and $\alpha$ respectively. We remark that $G_{\epsilon,\alpha}^{\beta, \beta\circ\gamma}$ contains only short edges. Hence $c_{\epsilon,\alpha}^{\beta, \beta\circ\gamma}=(-1)^m$ if $(i,i+1)\in G_{\epsilon,\alpha}^{\beta, \beta\circ\gamma}$ for all $1\le i<m$, otherwise the graph is disconnected and $c_{\epsilon,\alpha}^{\beta, \beta\circ\gamma}=0$. We summarize our discussion in the following theorem.
 \begin{theorem}  \label{thm:AntipodePR}
 Given $\alpha\in{\bf l}[n]$, in the Hopf algebra $PR$ we have
     $$S(\alpha)=\sum_{\gamma\in{\bf l}[n]}(-1)^{m}  d_{\alpha,\gamma}\,\, \gamma\,$$
     where $m=\ell\big( \epsilon\lor\gamma\big)$ and  $d_{\alpha,\gamma}$ is the number of  $\beta\in{\bf l}[n]$ such that for $\Lambda=\beta\lor(\beta\circ\gamma)$ we have
\begin{enumerate}
 \item[(i)] $\beta$ is increasing with respect to $\epsilon$ within each part of $\Lambda$,
 \item[(ii)] $\beta\circ\gamma$  is  increasing  with respect to $\alpha$ within each part of $\Lambda$, and
 \item[(iii)]  $\max_\epsilon(\Lambda_i)>_\epsilon\min_\epsilon(\Lambda_{i+1})$ or $\max_\alpha(\Lambda_i)>_\alpha \min_\alpha(\Lambda_{i+1})$ for all $1\le i<m$.
 \end{enumerate}
 \end{theorem}
 To our knowledge this theorem is new and it is a cancelation free formula.
\end{example}

\begin{example} For the monoid $\bf G$ with basis $\bf g$ in Section~\ref{ss:graph} the formula~\eqref{eq:SforKH} is  not cancelation free. However we can find another basis $\overline{\bf g}$
 that linearize  $\bf G$ such that the formula~\eqref{eq:SforKH} is cancelation free. More specifically for $x\in {\bf g}[I]$ a connected graph let
   $$\overline{x}=\sum_{\Phi\in{\bf \pi}[I]} (-1)^{|\Phi|-1}\big (|\Phi)-1\big)!\,\, x_\Phi,$$
   where  ${\bf \pi}[I]$ is the set of set partitions of $I$ and
  for $\Phi=\{A_1,A_2,\ldots,A_\ell\}$  we define $x_\Phi=x|_{A_1}x|_{A_2}\cdots x|_{A_\ell}$. The product  $x_\Phi$ is well defined since $\bf G$ is commutative.

   When $x$ is connected, we have that 
   $$\overline{x}=x+(\text{\it terms with more than 2 connected component}).$$ 
   This is not true if $x$ is not connected. 
   We leave to the reader the  exercise of showing that when $x$ is connected, we have
      $$\Delta_{A_1,A_2}(\overline{x})=0$$
      for any non-trivial decomposition $(A_1,A_2)\models I$.\footnote[1]{one needs to show and use the identity ${\displaystyle \sum_{k=0}^{\min(n,m)}}(-1)^{n+m-k-1}(n+m-k-1)! k! {n \choose k}{m \choose k} = 0$}
       That is to say $\overline{x}$ is primitive.
      
       If  $x\in {\bf g}[I]$  is not connected, then it decompose uniquely into connected component $x=x_1x_2\ldots x_m$ where $x_i$ is a connected subgraph on vertices $I_i\subseteq I$.
       Here $\{I_1,\ldots,I_m\}$ is a set partition of $I$. For such $x$, let us define $\overline{x}=\overline{x}_1\overline{x}_2\ldots\overline{x}_m$.
       now we get that
   $$\overline{x}=x+(\text{\it terms with more than $m+1$ connected components}).$$ 
 Hence the set $\{\overline{x}:x\in{\bf l}[I]\}$ form a basis of ${\bf G}[I]$.  In this basis, the multiplication is the same as before but the comultiplication is now
 $$
 \Delta_{A_1,A_2}(\overline{x}) = 
\begin{cases}
{\overline{x}|_{A_1}} \otimes {\overline{x}|_{A_2}}  & \text{if $A_1$ is the union of some of the parts of $\{I_1,\ldots,I_m\}$,} \\
0 & \text{otherwise.}
\end{cases}
$$
With this in hand, we now have a different basis  $\overline{\bf g}$ that linearized  $\bf G$ with a different comultiplication behavior.
With a reasonable amount of work similar to Example~\ref{ex:PR} and~\ref{ex:super}, the reader will find that  formula~\eqref{eq:SforKH} is also cancelation free in this case.
\end{example}
\subsection{Using Antipodes to derive new identities}\label{ss:antiStan} As we have seen in the introduction, any multiplicative morphism $\zeta\colon H\to\field$
gives rise to a combinatorial invariant $\chi=\phi_t\circ \Psi$ on $H$. For $x\in H_n$, the polynomials $\chi_x(t)$ encode combinatorial information about $x$ which depends on our choice of $\zeta$. 
Also, the combinatorial reciprocity $\chi_x(-1)=(\zeta\circ S)(x)$ is easily verified. 
\begin{example} For $H=\Kcb({\bf G})$ and for $x\in{\bf g}[n]$, let $\zeta(x)=1$ if $x$ is discrete, zero otherwise. In this case $\chi_x(t)$ is the chromatic polynomial of the graph $x$. Stanley's  $(-1)$-color theorem follows as
 $\pm \chi_x(-1)$ is the number of acyclic orientation of $x$.
\end{example}

The following example suggest a new venue to explore combinatorial identities using permutations.
\begin{example} Consider the Hopf algebra $PR=\Kc({\bf L})$ as studied in example~\ref{ex:PR} and let $\zeta(x)=1$ if $x=\epsilon$, and zero otherwise.
We have that $\zeta$ is indeed multiplicative. Since $PR$ is cocommutative then $\Psi\colon PR\to QSym$ will in fact be a symmetric function (see~\cite{ABS} for details). 
Here for $\alpha\in{\bf l}[n]$ we have
   $$ \Psi(\alpha)=\sum_{a\models n} c_a(\alpha) M_a, $$
where $a=(a_1,\ldots,a_\ell)\models n$ is an integer composition of $n$, and $c_a(\alpha)$ is the number of ways to decompose $\alpha$ into increasing subsequences of type $a$.
More precisely
  $$c_a(\alpha)=\big|\{ A\models [n] :  \text{for } 1\le i\le \ell,\  |A_i|=a_i \text{ and }\alpha|_{A_i} \text{ is increasing}\}\big|.$$
These numbers are studied in various place in mathematics and computer science. In particular Robinson-Schensted-Knuth(RSK) insertion shows that the coarsest possible $a$ for which $c_a(\alpha)\ne 0$
is a permutation of the shapes obtain via RSK (see \cite{permutation}).

The {\sl chromatic polynomial } $\chi_\alpha(t)$ is then
     $$ \chi_\alpha(t)=\sum_{a\models n} c_a(\alpha) {t \choose \ell(a)}. $$
 This polynomials, when evaluated at $t=m$ count the number of ways to color the entries of  the permutation $\alpha$ with at most $m$ distinct colors such that $\alpha$ restricted to a single color is increasing.
 Using Theorem~\ref{thm:AntipodePR} we get the identity
 \begin{equation}\label{eq:PRidentity}
     \sum_{a\models n} (-1)^{\ell(\alpha)}c_a(\alpha) =\chi_\alpha(-1)=\zeta\circ S(\alpha) = (-1)^n d_{\alpha,\epsilon}.
     \end{equation}
For any $\beta\in{\bf l}[n]$ and $\gamma=\epsilon$ in Theorem~\ref{thm:AntipodePR}, we have $\Lambda=\beta$ and the conditions (i) and (ii) are automatically satisfied.
Hence 
     $$d_{\alpha,\epsilon}=\big|\{\beta\in {\bf l}[n] : \beta_i>\beta_{i+1}\text{ or }\alpha^{-1}(\beta_i)>\alpha^{-1}(\beta_{i+1})\}\big|.$$
     The identity in Equation~\eqref{eq:PRidentity} relate combinatorial invariants for permutation that looks a priory unrelated. We summarize this in the following theorem
\begin{theorem} 
For $\alpha\in{\bf l}[n]$, the chromatic polynomial $\chi_\alpha(t)$ counts the number of ways to color increasing sequences of $\alpha$ with at most $t$ distinct colors.
We have the identity
  $$ \chi_\alpha(-1) = (-1)^n d_{\alpha,\epsilon},$$
  where $d_{\alpha,\epsilon}$ is the number of $\alpha$-decreasing order.
\end{theorem}
\end{example}

\begin{remark} Given $\alpha\in{\bf l}[n]$, one can associate a partial order $P_\alpha$ where $\alpha_i\prec \alpha_j$ if $i<j$ and $\alpha_i>\alpha_j$. 
As in~\cite{stanley} we can construct the incomparable graph $G_\alpha$ associated to $P_\alpha$. 
The symmetric function $\Psi(\alpha)$ above is in fact the Stanley chromatic symmetric function of the Graph $G_\alpha$.
A famous conjecture of Stanley~\cite{stanley} states that $\Psi(\alpha)$ is $e$-positive if  $P_\alpha$ is $({\bf 3}+ {\bf 1})$-avoiding.
In the language of permutations this is equivalent to say that $\alpha$ is $4123$ and $2341$ avoiding~\cite{Atkinson-Sagan}.
From the Hopf structure, one can see that it is natural to describe $e$ positivity in term of avoiding sequences.
What is surprising here is the fact that there should be finitely many and very simple patterns to avoid.
We also point out that here $d_{\alpha,\epsilon}$ also count the number of acyclic orientations of $G_\alpha$.
\end{remark}

The Hopf algebra $PR$ is free and generated by total orders that do not have any global ascent.
The free generators are $\{1,21,321,231,312,\ldots\}$. In the example above, we choose $\zeta$ to be $1$
on the generator $1$ and zero for all other generators. One can construct different $\zeta$'s by choosing any
subset of generators. This would lead to different coloring schemes and new identities with permutations.

\begin{example} Let us consider the case where $\zeta_{21}\colon PR\to\field$ is defined to be $1$ on the (free) generator $21$
and zero on the other. That is $\zeta(x)=1$ if $x=2143\ldots(2n)(2n-1)$, zero otherwise.
This defines a symmetric functions $\Psi_{21}\colon PR\to QSym$. 
Here for $\alpha\in{\bf l}[n]$ we have
   $$ \Psi_{21}(\alpha)=\sum_{a\models n} c'_a(\alpha) M_a, $$
where $a=(2a_1,\ldots,2a_\ell)\models n$ is an integer composition of $n$ with even parts, and $c_a(\alpha)$ is the number of ways to decompose $\alpha$ into $21^*$-subsequences of type $a$.
More precisely
  $$c'_a(\alpha)=\big|\{ A\models [n] :  \text{for } 1\le i\le \ell,\  |A_i|=2a_i \text{ and }st(\alpha|_{A_i})=2143\ldots(2a_i)(2a_i-1)\}\big|.$$
These numbers are new and strange but seem to have interesting properties to study.

The {\sl chromatic} polynomial $\chi^{21}_\alpha(t)$ is then
     $$ \chi^{21}_\alpha(t)=\sum_{a\models n} c'_a(\alpha) {t \choose \ell(a)}. $$
 This polynomials, when evaluated at $t=m$ counts the number of ways to colors the entries of $\alpha$ with at most $m$ distinct colors such that $\alpha$ restricted to a single color is a $21^*$-sequence.
 Using Theorem~\ref{thm:AntipodePR} we get the identity
 \begin{equation}\label{eq:PRidentity21}
     \sum_{a\models n} (-1)^{\ell(\alpha)}c'_a(\alpha) = (-1)^{n/2} d_{\alpha,2143\ldots(2n)(2n-1)}.
     \end{equation}
\end{example}

\begin{remark} 
The symmetric function $\Psi_{21}(\alpha)$ above is very different from the Stanley chromatic symmetric function for graphs.
However, we believe that using the Hopf structure, one can get some natural positivity using pattern avoidance.

\begin{conjecture} There is a finite set of permutations ${\mathcal A}$ such that for $\alpha\in{\bf l}[n]$
  $$(-1)^{n/2}\Psi_{21}(\alpha)(-h_1,-h_2,\ldots)$$
is $h$-positive for any $\alpha$ that is $\mathcal A$-avoiding. So far, our computer evidence suggests that ${\mathcal A}=\emptyset$.
\end{conjecture}
\end{remark}

%
%
%
%

\section{Future Work}\label{s:Future}

We now give a preview of our sequel work with J.~Machacek. Here we are interested in the Hopf monoid of hypergraphs as defined in Section~\ref{ss:hypergraph}

\subsection{Hypergraphical Nestohedron and antipode}\label{ss:nestohedron}
Similar to the construction of graphical zonotope, in this section we informally introduce a polytope associated to a hypergraph, the Hypergraphical Nestohedron.
Certain faces on the boundary of this polytope will naturally be labelled by acyclic orientation of the hypergraph. The antipode is then understood
as the signed sum of these special faces. 

\begin{definition} [Hypergraphical Nestohedron]
 Given a hypergraph $G$ on the vertex set $V=[n]$.
  the \emph{Hypergraphical Nestohedron} $P_G$ associated to $G$ is the polytope in ${\mathbb R}^n={\mathbb R}\{e_1,e_2,\ldots,e_n\}$
  defined by the Minkowsk sum
      $$P_G=\sum_{U\in G} {\bf \Delta}_U,$$
      where ${\bf \Delta}_U$ is the simplex given by the convex hull of the points $\{e_i : i\in U\}$.
\end{definition}

The acyclic orientations of $G$ actually label certain {\bf exterior} faces of $P_G$. Hence the coefficient of the discrete hypergraph in $S(G)$
is the homology of the complex labelled by the acyclic orientations of $G$. The other coefficients of $S(G)$ are also encoded in $P_G$.

For example, consider the hypergraph
$G=
\begin{tikzpicture}[scale=.7,baseline=.2cm]
	\node (1) at (0,.1) {$\scriptstyle 1$};
	\node (2) at (0,1) {$\scriptstyle 3$};
	\node (3) at (1,.5) {$\scriptstyle 2$};
	\draw [fill=blue] (.1,.1) .. controls (.25,.5) .. (.15,.9) .. controls (.4,.5) .. (.85,.5) .. controls (.5,.4) .. (.1,.1) ; 
	\red{\draw (.9,.7).. controls (0.6,1) ..(0.2,1);}
\end{tikzpicture} 
$.
We  have 
$ P_G=\blue{{\bf \Delta}_{123}} + \red{{\bf \Delta}_{23} }
$
where \vskip -18pt
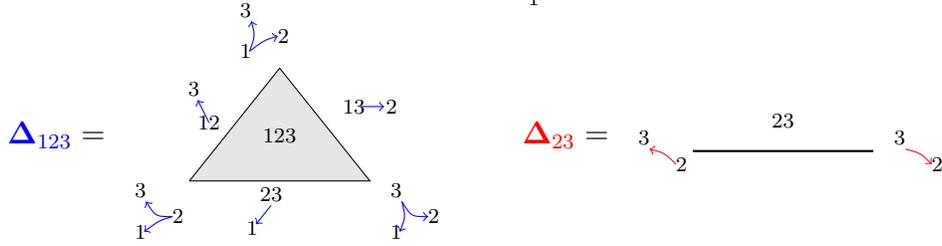
$$\blue{{\bf \Delta}_{123}} = \begin{tikzpicture}[scale=2,baseline=.5cm]
	\node (1) at (0.5,1.0) {
	     $\begin{tikzpicture}[scale=.5,baseline=.5cm]
		\node (1) at (0,.1) {$\scriptstyle 1$};
		\node (2) at (0,1.2) {$\scriptstyle 3$};
		\node (3) at (1,.5) {$\scriptstyle 2$};
		\blue{\draw [->] (.1,.1) .. controls (.25,.5) .. (.15,.9) ; 
		\draw [->] (.1,.1) .. controls (.5,.4) .. (.85,.5) ; }
		\end{tikzpicture} $};
	\node (2) at (-.2,-.2) {
	     $\begin{tikzpicture}[scale=.5,baseline=.5cm]
		\node (1) at (0,.1) {$\scriptstyle 1$};
		\node (2) at (0,1.2) {$\scriptstyle 3$};
		\node (3) at (1,.5) {$\scriptstyle 2$};
		\blue{\draw [->] (.85,.5) .. controls (.4,.5) .. (.15,.9) ; 
		\draw [->] (.85,.5) .. controls (.5,.4) .. (.1,.1) ; }
		\end{tikzpicture} $};
	\node (3) at (1.5,-.2) {
	     $\begin{tikzpicture}[scale=.5,baseline=.5cm]
		\node (1) at (0,.1) {$\scriptstyle 1$};
		\node (2) at (0,1.2) {$\scriptstyle 3$};
		\node (3) at (1,.5) {$\scriptstyle 2$};
		\blue{\draw [->] (.15,.9) .. controls (.4,.5) .. (.85,.5) ; 
		\draw [->] (.15,.9) .. controls (.25,.5) .. (.1,.1) ; }
		\end{tikzpicture} $};
	\node at (1.2,.5) {
	     $\begin{tikzpicture}[scale=.5,baseline=.5cm]
		\node (1) at (0,.5) {$\scriptstyle 13$};
		\node (3) at (1,.5) {$\scriptstyle 2$};
		\blue{\draw [->] (.2,.5) -- (.8,.5) ; }
		\end{tikzpicture} $};
	\node at (.1,.5) {
	     $\begin{tikzpicture}[scale=.5,baseline=.5cm]
		\node (1) at (.5,.1) {$\scriptstyle 12$};
		\node (3) at (.1,1) {$\scriptstyle 3$};
		\blue{\draw [->] (.5,.1) -- (.2,.7) ; }
		\end{tikzpicture} $};
	\node at (.5,-.2) {
	     $\begin{tikzpicture}[scale=.5,baseline=.5cm]
		\node (1) at (0,.1) {$\scriptstyle 1$};
		\node (3) at (.5,1) {$\scriptstyle 23$};
		\blue{\draw [->] (.5,.7) --(.1,.2) ; }
		\end{tikzpicture} $};
	\draw [fill=gray!20!white] (0,0) -- (.6,.75) -- (1.2,0) --(0,0) ; 
	\node at (0.6,.3) {$\scriptstyle 123$};
\end{tikzpicture} \qquad
\red{{\bf \Delta}_{23}} = \begin{tikzpicture}[scale=2,baseline=.1cm]
	\node at (-.2,0) {
	     $\begin{tikzpicture}[scale=.5,baseline=.5cm]
		\node (2) at (0,1.2) {$\scriptstyle 3$};
		\node (3) at (1,.5) {$\scriptstyle 2$};
		\red{\draw [->] (.85,.5) .. controls (.55,.8) .. (.15,.9) ; }
		\end{tikzpicture} $};
	\node at (1.5,0) {
	     $\begin{tikzpicture}[scale=.5,baseline=.5cm]
		\node (2) at (0,1.2) {$\scriptstyle 3$};
		\node (3) at (1,.5) {$\scriptstyle 2$};
		\red{\draw [<-] (.85,.5) .. controls (.55,.8) .. (.15,.9) ; }
		\end{tikzpicture} $};
	\draw [thick] (0,0) -- (1.2,0); 
	\node at (.6,.2) {$\scriptstyle 23$};
\end{tikzpicture} 
$$
Here we have drawn the orientation of each edge on the boundary of its corresponding simplex.
The interior of the simplex is labeled by the full set and corresponds to contraction of the edges.
We then  label the faces of the sum $P_G$ as follow. For each face of $P_G$, we consider how it is obtained in the sum 
and this gives us an orientation of a contraction of $G$. We then contract further any cycle of this orientation and 
obtain\footnote{it is not obvious that this is well defined, this is the main part of our future work} an acyclic orientation of a contraction of $G$:  
$$P_G=\begin{tikzpicture}[scale=2,baseline=.5cm]
	\node (1) at (0.5,1.0) {
	     $\begin{tikzpicture}[scale=.5,baseline=.5cm]
		\node (1) at (0,.1) {$\scriptstyle 1$};
		\node (2) at (0,1.2) {$\scriptstyle 3$};
		\node (3) at (1,.5) {$\scriptstyle 2$};
		\blue{\draw [->] (.1,.1) .. controls (.25,.5) .. (.15,.9) ; 
		\draw [->] (.1,.1) .. controls (.5,.4) .. (.85,.5) ; }
		\red{\draw [->] (.85,.5) .. controls (.55,.8) .. (.25,.9) ; }
		\end{tikzpicture} $};
	\node  at (2,1.0) {
	     $\begin{tikzpicture}[scale=.5,baseline=.5cm]
		\node (1) at (0,.1) {$\scriptstyle 1$};
		\node (2) at (0,1.2) {$\scriptstyle 3$};
		\node (3) at (1,.5) {$\scriptstyle 2$};
		\blue{\draw [->] (.1,.1) .. controls (.25,.5) .. (.15,.9) ; 
		\draw [->] (.1,.1) .. controls (.5,.4) .. (.85,.5) ; }
		\red{\draw [<-] (.85,.6) .. controls (.55,.8) .. (.15,.9) ; }
		\end{tikzpicture} $};
	\node (2) at (-.2,-.2) {
	     $\begin{tikzpicture}[scale=.5,baseline=.5cm]
		\node (1) at (0,.1) {$\scriptstyle 1$};
		\node (2) at (0,1.2) {$\scriptstyle 3$};
		\node (3) at (1,.5) {$\scriptstyle 2$};
		\blue{\draw [->] (.85,.5) .. controls (.4,.5) .. (.15,.9) ; 
		\draw [->] (.85,.5) .. controls (.5,.4) .. (.1,.1) ; }
		\red{\draw [->] (.85,.5) .. controls (.55,.8) .. (.25,.9) ; }
		\end{tikzpicture} $};
	\node (3) at (2.7,-.2) {
	     $\begin{tikzpicture}[scale=.5,baseline=.5cm]
		\node (1) at (0,.1) {$\scriptstyle 1$};
		\node (2) at (0,1.2) {$\scriptstyle 3$};
		\node (3) at (1,.5) {$\scriptstyle 2$};
		\blue{\draw [->] (.15,.9) .. controls (.4,.5) .. (.85,.5) ; 
		\draw [->] (.15,.9) .. controls (.25,.5) .. (.1,.1) ; }
		\red{\draw [<-] (.85,.6) .. controls (.55,.8) .. (.15,.9) ; }
		\end{tikzpicture} $};
	\node at (2.4,.5) {
	     $\begin{tikzpicture}[scale=.5,baseline=.5cm]
		\node (1) at (0,.5) {$\scriptstyle 13$};
		\node (3) at (1,.5) {$\scriptstyle 2$};
		\blue{\draw [->] (.2,.5) -- (.8,.5) ; }
		\red{\draw [->] (.2,.5) .. controls (.55,.8) .. (.8,.6) ; }
		\end{tikzpicture} $};
	\node at (.1,.5) {
	     $\begin{tikzpicture}[scale=.5,baseline=.5cm]
		\node (1) at (.5,.1) {$\scriptstyle 12$};
		\node (3) at (.1,1) {$\scriptstyle 3$};
		\blue{\draw [->] (.5,.1) -- (.2,.7) ; }
		\red{\draw [->] (.5,.1) .. controls (.5,.6) .. (.3,.7) ; }
		\end{tikzpicture} $};
	\node at (1.2,-.2) {
	     $\begin{tikzpicture}[scale=.5,baseline=.5cm]
		\node (1) at (0,.1) {$\scriptstyle 1$};
		\node (3) at (.5,1) {$\scriptstyle 23$};
		\blue{\draw [->] (.5,.7) --(.1,.2) ; }
		\end{tikzpicture} $};
	\node at (1.2,1) {
	     $\begin{tikzpicture}[scale=.5,baseline=.5cm]
		\node (1) at (0,.1) {$\scriptstyle 1$};
		\node (3) at (.5,1) {$\scriptstyle 23$};
		\blue{\draw [<-] (.5,.7) --(.1,.2) ; }
		\end{tikzpicture} $};
	\draw [fill=gray!10!white] (0,0) -- (.6,.75) -- (1.8,.75) -- (2.4,0) --(0,0) ; 
	\node at (1.2,.35) {$\scriptstyle 123$};
\end{tikzpicture}
$$
 For example  
$\begin{tikzpicture}[scale=.5,baseline=.3cm]
		\node (1) at (0,.1) {$\scriptstyle 1$};
		\node (2) at (0,1.2) {$\scriptstyle 3$};
		\node (3) at (1,.5) {$\scriptstyle 2$};
		\blue{\draw [->] (.15,.9) .. controls (.4,.5) .. (.85,.5) ; 
		\draw [->] (.15,.9) .. controls (.25,.5) .. (.1,.1) ; }
		\end{tikzpicture} + 
\begin{tikzpicture}[scale=.5,baseline=.3cm]
		\node (2) at (0,1.2) {$\scriptstyle 3$};
		\node (3) at (1,.5) {$\scriptstyle 2$};
		\red{\draw [->] (.85,.5) .. controls (.55,.8) .. (.15,.9) ; }
		\end{tikzpicture}  = 
\begin{tikzpicture}[scale=.5,baseline=.3cm]
		\node (1) at (0,.1) {$\scriptstyle 1$};
		\node (2) at (0,1.2) {$\scriptstyle 3$};
		\node (3) at (1,.5) {$\scriptstyle 2$};
		\blue{\draw [->] (.15,.9) .. controls (.4,.5) .. (.85,.5) ; 
		\draw [->] (.15,.9) .. controls (.25,.5) .. (.1,.1) ; }
		\red{\draw [->] (.85,.5) .. controls (.55,.8) .. (.15,.9) ; }
		\end{tikzpicture}
		\quad\mapsto\quad
\begin{tikzpicture}[scale=.5,baseline=.3cm]
		\node (1) at (0,.1) {$\scriptstyle 1$};
		\node (3) at (.5,1) {$\scriptstyle 23$};
		\blue{\draw [->] (.5,.7) --(.1,.2) ; }
		\end{tikzpicture} 
$ 
is in the face 
$\begin{tikzpicture}[scale=.5,baseline=.3cm]
		\node (1) at (0,.1) {$\scriptstyle 1$};
		\node (3) at (.5,1) {$\scriptstyle 23$};
		\blue{\draw [->] (.5,.7) --(.1,.2) ; }
		\end{tikzpicture} $. 
The reader can check that all points in this face will get the same labeling regardless on how it is obtained.

The antipode for $G$ contains 3 terms. 
  $$S(G) = -
  \begin{tikzpicture}[scale=.7,baseline=.2cm]
	\node (1) at (0,.1) {$\scriptstyle 1$};
	\node (2) at (0,1) {$\scriptstyle 3$};
	\node (3) at (1,.5) {$\scriptstyle 2$};
	\draw [fill=blue] (.1,.1) .. controls (.25,.5) .. (.15,.9) .. controls (.4,.5) .. (.85,.5) .. controls (.5,.4) .. (.1,.1) ; 
	\red{\draw (.9,.7).. controls (0.6,1) ..(0.2,1);}
\end{tikzpicture} 
 + 2\big(\begin{tikzpicture}[scale=.7,baseline=.2cm]
	\node (1) at (0,.1) {$\scriptstyle 1$};
	\node (2) at (0,1) {$\scriptstyle 3$};
	\node (3) at (1,.5) {$\scriptstyle 2$};
		\red{\draw (.9,.7).. controls (0.6,1) ..(0.2,1);}
\end{tikzpicture} \big) - 2 \big(\begin{tikzpicture}[scale=.7,baseline=.2cm]
	\node (1) at (0,.1) {$\scriptstyle 1$};
	\node (2) at (0,1) {$\scriptstyle 3$};
	\node (3) at (1,.5) {$\scriptstyle 2$};
\end{tikzpicture} 
\big)
$$
The coefficient of the discrete graph is the sum of the six acyclic orientations that
corresponding to the three faces on the left and the three faces on the right. We call these exterior faces as no contraction occurs. The total homology is $2$ in this case.
The coefficient $+2$ in $S(G)$ corresponds to the two horizontal faces in the picture (only $\{2,3\}$ is contracted). Finally the coefficient $-1$ corresponds the interior face of the polytope ($\{1,2,3\}$ is contracted).

In our coming work with Machacek, we  will formalize this and show how the hypergraphical nestohedron contains
all the information of the antipode for hypergraphs.

\subsection{Generalization of Shareshian-Wachs quasisymmetric function}\label{ss:non-com}

For a commutative and cocommutative linearized Hopf monoid $\bH$, the projection map ${\bf L}\times \bH\to\bH$ is a morphism of Hopf monoid. The functor $\Kcb$ give us a Hopf morphism
$\Kc(\bH)=\Kcb({\bf L}\times \bH)\to\Kcb(\bH)$. In~\cite{Guay-Paquet}, M. Guay-Paquet  defines a $q$-deformation
of the comultication of ${\bf L}\times {\bf G}$ where $\bf G$ is as in Section~\ref{ss:graph}. Let us denote by 
$({\bf L}\times {\bf G})_q$ this new monoid structure. At $q=1$ it is clear from~\cite{Guay-Paquet} that we obtain the usual
structure on ${\bf L}\times {\bf G}$. Also, the map $\zeta$ defined by  $\zeta(g)= 1$ if $g$ is discrete graph, $0$ otherwise, is multiplicative on  $\Kc(\bH)$, $\Kcb(\bH)$ and $\Kcb({\bf L}\times {\bf G})_q$.
It commutes with the morphism $\Kc(\bH)\to\Kcb(\bH)$ above,
hence at $q=1$ the invariant defined on both $\Kc(\bH)$ and $\Kcb(\bH)$ are the same invariant (the chromatic polynomial) but 
it is $q$-deformed on $\Kcb({\bf L}\times {\bf G})_q$. We will exploit this similarity to define Shareshian-Wachs $q$-deformation
of other combinatorial invariants.

\small
\bibliographystyle{abbrv}  
\bibliography{AntipodeLinHopfMono}

\end{document}